\documentclass[a4paper,12pt]{amsart}

\usepackage{amsmath,amsxtra,amssymb,latexsym, amscd,amsthm,dsfont,subcaption,cases}
\usepackage[mathscr]{eucal}
\usepackage{mathrsfs,bbm,enumerate,calc,yhmath,mathtools,faktor}
\usepackage{stackengine}
\usepackage{pict2e}
\usepackage{tikz}
\usepackage{graphicx}
\usepackage{color}
\usepackage[foot]{amsaddr}

\usepackage[a4paper]{geometry}
\usepackage{hyperref}
\hypersetup{colorlinks=true, breaklinks=true, urlcolor= blue, linkcolor= black, citecolor=teal,
  pdftitle={Friedman-Ramanujan functions}, 
  pdfauthor={Nalini Anantharaman and Laura Monk}, 
  pdfsubject={}}

\usepackage[capitalize]{cleveref}
\crefname{equation}{equation}{equations}
\Crefname{equation}{Equation}{Equations}
\crefname{figure}{Figure}{Figures}
\Crefname{figure}{Figure}{Figures}

\numberwithin{equation}{section}

\title{Typical hyperbolic surfaces have a spectral gap greater than $2/9 - \epsilon$}

\author{Nalini Anantharaman\textsuperscript{1} and Laura Monk\textsuperscript{2}}

\address[1]{Coll\`ege de France, 11 place Marcelin Berthelot, 75005 Paris / IRMA, 7 rue Ren\'e Descartes, 67084 Strasbourg Cedex, France} 

\address[2]{School of Mathematics, University of Bristol, Bristol BS8 1UG, U.K.}

\email{nalini.anantharaman@college-de-france.fr}
\email{laura.monk@bristol.ac.uk}

\makeatletter
\@namedef{subjclassname@2020}{\textup{2020} Mathematics Subject Classification}
\makeatother



\date{\today}

\theoremstyle{plain}
\newtheorem{thm}{Theorem}[section]
\newtheorem{prp}[thm]{Proposition}

\newtheorem{lem}[thm]{Lemma}

\newtheorem*{namedthm}{\namedthmname}
\newcounter{namedthm}



\theoremstyle{definition}

\newtheorem{rem}[thm]{Remark}

\newtheorem{nota}[thm]{Notation}

\renewcommand{\d}{\, \mathrm{d}}

\makeatletter
\newcommand*{\ov}[1]{%
  $\m@th\overline{\mbox{#1}}$%
}
\newcommand*{\ovA}[1]{%
  $\m@th\overline{\mbox{#1}\raisebox{3mm}{}}$%
}
\newcommand*{\ovB}[1]{%
  $\m@th\overline{\mbox{#1\rule{0pt}{3mm}}}$%
}
\newcommand*{\ovC}[1]{%
  $\m@th\overline{\mbox{#1\strut}}$%
}
\newcommand*{\ovD}[1]{%
  $\m@th\overline{\mbox{#1\vphantom{\"A}}}$%
}
\newcommand*{\ovE}[1]{%
  $\m@th\overline{\raisebox{0pt}[1.2\height]{#1}}$%
}
\newcommand*{\ovF}[1]{%
  $\m@th\overline{\raisebox{0pt}[\dimexpr\height+0.3mm\relax]{#1}}$%
}
\newcommand*{\ovG}[1]{%
  $\m@th\overline{\raisebox{0pt}[\dimexpr\height+1mm\relax]{#1\vphantom{A}}}$%
}
\makeatother

\newcommand{\Z}{\mathbb{Z}}

\newcommand{\R}{\mathbb{R}}
\newcommand{\C}{\mathbb{C}}

\newcommand{\D}{\mathcal{D}}

\DeclareSymbolFont{extraup}{U}{zavm}{m}{n}
\DeclareMathSymbol{\varheart}{\mathalpha}{extraup}{86}
\DeclareMathSymbol{\vardiamond}{\mathalpha}{extraup}{87}

\newcommand{\nwc}{\newcommand}

\nwc{\mf}{\mathbf} 
\nwc{\blds}{\boldsymbol} 
\nwc{\ml}{\mathcal} 

\nwc{\lam}{\lambda}
\nwc{\del}{\delta}
\nwc{\Del}{\Delta}
\nwc{\Lam}{\Lambda}
\nwc{\elll}{\ell}

\nwc{\IA}{\mathbb{A}} 
\nwc{\IB}{\mathbb{B}} 
\nwc{\IC}{\mathbb{C}} 
\nwc{\ID}{\mathbb{D}} 
\nwc{\IE}{\mathbb{E}} 
\nwc{\IF}{\mathbb{F}} 
\nwc{\IG}{\mathbb{G}} 
\nwc{\IH}{\mathbb{H}} 
\nwc{\IN}{\mathbb{N}} 
\nwc{\IP}{\mathbb{P}} 
\nwc{\IQ}{\mathbb{Q}} 
\nwc{\IR}{\mathbb{R}} 
\nwc{\IS}{\mathbb{S}} 
\nwc{\IT}{\mathbb{T}} 
\nwc{\IZ}{\mathbb{Z}} 
\def\bbbone{{\mathchoice {1\mskip-4mu {\rm{l}}} {1\mskip-4mu {\rm{l}}}
{ 1\mskip-4.5mu {\rm{l}}} { 1\mskip-5mu {\rm{l}}}}}
\def\bbleft{{\mathchoice {[\mskip-3mu {[}} {[\mskip-3mu {[}}{[\mskip-4mu {[}}{[\mskip-5mu {[}}}}
\def\bbright{{\mathchoice {]\mskip-3mu {]}} {]\mskip-3mu {]}}{]\mskip-4mu {]}}{]\mskip-5mu {]}}}}
\nwc{\setK}{\bbleft 1,K \bbright}
\nwc{\setN}{\bbleft 1,\cN \bbright}
 
\nwc{\va}{{\bf a}}
\nwc{\vb}{{\bf b}}
\nwc{\vc}{{\bf c}}
\nwc{\vd}{{\bf d}}
\nwc{\ve}{{\bf e}}
\nwc{\vf}{{\bf f}}
\nwc{\vg}{{\bf g}}
\nwc{\vh}{{\bf h}}
\nwc{\vi}{{\bf i}}
\nwc{\vI}{{\bf I}}
\nwc{\vj}{{\bf j}}
\nwc{\vk}{{\bf k}}
\nwc{\vl}{{\bf l}}
\nwc{\vm}{{\bf m}}
\nwc{\vM}{{\bf M}}
\nwc{\vn}{{\bf n}}
\nwc{\vo}{{\it o}}
\nwc{\vp}{{\bf p}}
\nwc{\vq}{{\bf q}}
\nwc{\vr}{{\bf r}}
\nwc{\vs}{{\bf s}}
\nwc{\vt}{{\bf t}}
\nwc{\vu}{{\bf u}}
\nwc{\vv}{{\bf v}}
\nwc{\vw}{{\bf w}}
\nwc{\vx}{{\bf x}}
\nwc{\vy}{{\bf y}}
\nwc{\vz}{{\bf z}}
\nwc{\bal}{\blds{\alpha}}
\nwc{\bep}{\blds{\epsilon}}
\nwc{\barbep}{\overline{\blds{\epsilon}}}
\nwc{\bnu}{\blds{\nu}}
\nwc{\bmu}{\blds{\mu}}
\nwc{\bet}{\blds{\eta}}

\nwc{\bk}{\blds{k}}
\nwc{\bm}{\blds{m}}
\nwc{\bM}{\blds{M}}
\nwc{\bp}{\blds{p}}
\nwc{\bq}{\blds{q}}
\nwc{\bn}{\blds{n}}
\nwc{\bv}{\blds{v}}
\nwc{\bw}{\blds{w}}
\nwc{\bx}{\blds{x}}
\nwc{\bxi}{\blds{\xi}}
\nwc{\by}{\blds{y}}
\nwc{\bz}{\blds{z}}

\nwc{\cA}{\ml{A}}
\nwc{\cB}{\ml{B}}
\nwc{\cC}{\ml{C}}
\nwc{\cD}{\ml{D}}
\nwc{\cE}{\ml{E}}
\nwc{\cF}{\ml{F}}
\nwc{\cG}{\ml{G}}
\nwc{\cH}{\ml{H}}
\nwc{\cI}{\ml{I}}
\nwc{\cJ}{\ml{J}}
\nwc{\cK}{\ml{K}}
\nwc{\cL}{\ml{L}}
\nwc{\cM}{\ml{M}}
\nwc{\cN}{\ml{N}}
\nwc{\cO}{\ml{O}}
\nwc{\cP}{\ml{P}}
\nwc{\cQ}{\ml{Q}}
\nwc{\cR}{\ml{R}}
\nwc{\cS}{\ml{S}}
\nwc{\cT}{\ml{T}}
\nwc{\cU}{\ml{U}}
\nwc{\cV}{\ml{V}}
\nwc{\cW}{\ml{W}}
\nwc{\cX}{\ml{X}}
\nwc{\cY}{\ml{Y}}
\nwc{\cZ}{\ml{Z}}

\nwc{\fA}{\mathfrak{a}}
\nwc{\fB}{\mathfrak{b}}
\nwc{\fC}{\mathfrak{c}}
\nwc{\fD}{\mathfrak{d}}
\nwc{\fE}{\mathfrak{e}}
\nwc{\fF}{\mathfrak{f}}
\nwc{\fG}{\mathfrak{g}}
\nwc{\fH}{\mathfrak{h}}
\nwc{\fI}{\mathfrak{i}}
\nwc{\fJ}{\mathfrak{j}}
\nwc{\fK}{\mathfrak{k}}
\nwc{\fL}{\mathfrak{l}}
\nwc{\fM}{\mathfrak{m}}
\nwc{\fN}{\mathfrak{n}}
\nwc{\fO}{\mathfrak{o}}
\nwc{\fP}{\mathfrak{p}}
\nwc{\fQ}{\mathfrak{q}}
\nwc{\fR}{\mathfrak{r}}
\nwc{\fS}{\mathfrak{s}}
\nwc{\fT}{\mathfrak{t}}
\nwc{\fU}{\mathfrak{u}}
\nwc{\fV}{\mathfrak{v}}
\nwc{\fW}{\mathfrak{w}}
\nwc{\fX}{\mathfrak{x}}
\nwc{\fY}{\mathfrak{y}}
\nwc{\fZ}{\mathfrak{z}}

\nwc{\tA}{\widetilde{A}}
\nwc{\tB}{\widetilde{B}}
\nwc{\tE}{E^{\vareps}}
\nwc{\tk}{\tilde k}
\nwc{\tN}{\tilde N}
\nwc{\tP}{\widetilde{P}}
\nwc{\tQ}{\widetilde{Q}}
\nwc{\tR}{\widetilde{R}}
\nwc{\tV}{\widetilde{V}}
\nwc{\tW}{\widetilde{W}}
\nwc{\ty}{\tilde y}
\nwc{\teta}{\tilde \eta}
\nwc{\tdelta}{\tilde \delta}
\nwc{\tlambda}{\tilde \lambda}
\nwc{\ttheta}{\tilde \theta}
\nwc{\tvartheta}{\tilde \vartheta}
\nwc{\tPhi}{\widetilde \Phi}
\nwc{\tpsi}{\tilde \psi}
\nwc{\tmu}{\tilde \mu}

\nwc{\To}{\longrightarrow} 

\nwc{\ad}{\rm ad}
\nwc{\eps}{\epsilon}
\nwc{\ep}{\epsilon}
\nwc{\vareps}{\varepsilon}

\def\ep{\epsilon}

\def\sq2{\sqrt{2}}

\def\t2{{\mathbb T}^2}
\def\s2{{\mathbb S}^2}

\def\R{\mathbb{R}}

\def\Z{\mathbb{Z}}
\def\C{\mathbb{C}}
\def\O{\mathcal{O}}

\nwc{\lap}{\bigtriangleup}
\nwc{\rest}{\restriction}
\nwc{\Diff}{\operatorname{Diff}}
\nwc{\diam}{\operatorname{diam}}
\nwc{\Res}{\operatorname{Res}}
\nwc{\Spec}{\operatorname{Spec}}
\nwc{\Vol}{\operatorname{Vol}}
\nwc{\Op}{\operatorname{Op}}
\nwc{\supp}{\operatorname{supp}}
\nwc{\Span}{\operatorname{span}}

\nwc{\dia}{\varepsilon}
\nwc{\cut}{f}
\nwc{\qm}{u_\hbar}

\def\hto0{\xrightarrow{\hbar\to 0}}

\def\rto0{\xrightarrow{r\to 0}}

\providecommand{\norm}[1]{\lVert#1\rVert}

\nwc{\la}{\langle}
\nwc{\ra}{\rangle}
\nwc{\lp}{\left(}
\nwc{\rp}{\right)}

\nwc{\bequ}{\begin{equation}}
\nwc{\be}{\begin{equation}}
\nwc{\ben}{\begin{equation*}}
\nwc{\bea}{\begin{eqnarray}}
\nwc{\bean}{\begin{eqnarray*}}
\nwc{\bit}{\begin{itemize}}
\nwc{\bver}{\begin{verbatim}}

\nwc{\eequ}{\end{equation}}
\nwc{\ee}{\end{equation}}
\nwc{\een}{\end{equation*}}
\nwc{\eea}{\end{eqnarray}}
\nwc{\eean}{\end{eqnarray*}}
\nwc{\eit}{\end{itemize}}
\nwc{\ever}{\end{verbatim}}

\makeatletter
\newlength{\temp@wc@width}
\newlength{\temp@wc@height}
\newcommand{\widecheck}[1]{%
  \setlength{\temp@wc@width}{\widthof{$#1$}}%
  \setlength{\temp@wc@height}{\heightof{$#1$}}%
  #1\hspace{-\temp@wc@width}%
  \raisebox{\temp@wc@height+2pt}[\heightof{$\widehat{#1}$}]%
     {\rotatebox[origin=c]{180}{\vbox to 0pt{\hbox{$\widehat{\hphantom{#1}}$}}}}%
}
\makeatother

\newcommand{\Volwp}[1][g]{\mathrm{Vol}_{#1}^{\mathrm{\scriptsize{WP}}}}
\newcommand{\Pwpo}{\mathbb{P}_g^{\mathrm{\scriptsize{WP}}}}
\newcommand{\Ewpo}[1][g]{\mathbb{E}_{#1}^{\mathrm{\scriptsize{WP}}}}
\newcommand{\Pwp}[1]{\Pwpo \left( #1 \right)}
\newcommand{\Ewp}[2][g]{\mathbb{E}_{#1}^{\mathrm{\scriptsize{WP}}} \Bigg[ #2 \Bigg]}
\newcommand{\Ewpon}[1][g, n]{\mathbb{E}_{#1}^{\mathrm{\scriptsize{WP}}}}

\DeclarePairedDelimiter{\paren}{(}{)}

\DeclarePairedDelimiter{\abso}{|}{|}
\DeclarePairedDelimiter{\brac}{[}{]}

\DeclarePairedDelimiter{\norminf}{\|}{\|_{\infty}}
\let\div\relax
\newcommand{\div}[1]{\paren*{\frac{#1}{2}}}

\renewcommand{\O}[2][ ]{\mathcal{O}_{#1} \left( #2 \right)}

\newcommand{\1}[1]{\mathds{1}_{#1}}

\newcommand{\av}[2][\mathrm{all}]{\langle #2 \rangle_g^{{#1}}}
\newcommand{\avQ}[2][\mathrm{all}]{\langle #2 \rangle_{g,Q}^{{#1}}}
\newcommand{\avb}[2][\mathrm{all}]{\left\langle #2 \right\rangle_g^{{#1}}}

\newcommand{\FRrem}{\mathcal{R}}

\newcommand{\eqc}[1]{[ #1 ]_{\mathrm{loc}}}

\newcommand{\x}{\mathbf{x}}
\newcommand{\y}{\mathbf{y}}

\makeatletter
\newcommand{\smallbullet}{} 
\DeclareRobustCommand\smallbullet{%
  \mathord{\mathpalette\smallbullet@{0.5}}%
}
\newcommand{\smallbullet@}[2]{%
  \, \vcenter{\hbox{\scalebox{#2}{$\m@th#1\bullet$}}} \,%
}
\makeatother

\newcommand{\Sf}{\mathbf{S}}
\newcommand{\type}{\mathbf{T}}

\newcommand{\ord}{N}

\newcommand{\curve}{\mathbf{c}}

\newcommand{\tfL}{R}
\newcommand{\Atf}{\mathrm{TF}_g^{\kappa,\tfL}}
\newcommand{\Ntang}[1][ ]{N_{\tfL #1}^{\mathrm{tang}}}
\newcommand{\Ninj}[1][ ]{N_{\kappa #1}^{\mathrm{inj}}}
\newcommand{\chic}{{\chi_+}}
\newcommand{\MC}{\mathrm{MC}}

\begin{document}

\maketitle

\begin{abstract}
  In this article, we prove that typical hyperbolic surfaces, sampled with the Weil--Petersson
  probability measure, have a spectral gap at least $2/9 - \epsilon$. This is an intermediate result
  on the way to our proof of the optimal spectral gap $1/4 - \epsilon$, building on the results of
  the first part of this series, \cite{Ours1}. A significant part of the proof is an explicit
  inclusion-exclusion argument to exclude tangles at the level of precision $1/g$.
\end{abstract}

\tableofcontents

\section{Introduction}
\label{sec:introduction}

This article is part of our series of work \cite{Ours1,Moebius,Ours2}, ultimately proving that the
spectral gap of typical hyperbolic surfaces is near-optimal, i.e.
\begin{equation*}
  \forall \epsilon > 0, \quad
  \lim_{g \rightarrow \infty}
    \Pwp{\lambda_1 \geq \frac 14 - \epsilon} = 1
\end{equation*}
where $\mathbb{P}_g^{\mathrm{WP}}$ is the Weil--Petersson probability on the moduli space of
hyperbolic surfaces of genus~$g$. This shows that typical hyperbolic surfaces have an optimal
spectral gap, due to a result of Huber \cite{huber1974}. This conclusion is reached at the end of
\cite{Ours2}, by proving a Friedman--Ramanujan type result for the volume functions associated to
any fixed type of closed geodesics.

This shorter paper shows how the results from \cite{Ours1} already allow to prove the following partial
result, namely:
\begin{thm}
  \label{thm:29}
  For any $\epsilon > 0$,
  \begin{equation*}
    \lim_{g \rightarrow \infty}
    \Pwp{\lambda_1 \geq \frac 29 - \epsilon} = 1.
  \end{equation*}
\end{thm}
This beats the previous record of $3/16-\epsilon$ by Wu--Xue \cite{wu2022} and
Lipnowski--Wright~\cite{lipnowski2021}, but does not yet reach the optimal~$1/4 - \epsilon$
proven in the follow-up article~\cite{Ours2}. The reason behind this incremental progress is that
we perform asymptotic expansions in powers of $1/g$, the $3/16$ result corresponding to the
leading order $1/g^0$, the~$2/9$ to the second order $1/g$, and the $1/4$ to arbitrary
orders $1/g^N$.

We prove \cref{thm:29} by a classic trace method, as sketched in \cite[Section~3]{Ours1}, performing
all computations at the second order in the expansion in powers of $1/g$. Our computations are
explicit and we gain a precise understanding of the problem at this level of precision.

Since this article is in direct continuation with \cite{Ours1}, we will not re-introduce the
notations and results it contains. The main novel challenge here is that we need to make a
probabilistic hypothesis in order to \emph{remove tangles}, which requires to adapt the methods
developed in \cite{Ours1} to more general averages that include an indicator function. This work is
done in Sections~\ref{sec:incl-excl-tangle} to \ref{sec:reduct-numb-topol}. We then conclude to the
proof of \cref{thm:29} in \cref{sec:proof_29}.

In the final proof of the optimal result \cite{Ours2}, we use a different, more abstract approach to
the removal of tangles, based on the Moebius formula \cite{Moebius}. We believe it is interesting to
present these different approaches to remove tangles, and how this removal can be done explicitely
at this level of precision.

\subsection*{Acknowledgements}

The authors would like to express their gratitude to Joel Friedman, for explaining his proof of
Alon’s conjecture to us in detail, which lead to significant advances in our project. We would also
like to thank Yuhao Xue for useful comments on a first version of the paper. This research was funded by
the guest program of the Max-Planck Institute for Mathematics during the year 2021-2022, the EPSRC
grant EP/W007010/1 since 2022, the prize L’Oréal-UNESCO Young Talents France for Women in Science,
and by the European Research Council (ERC) under the European Union’s Horizon 2020 research and
innovation programme (Grant agreement No. 101096550).

\section{The tangle-free hypothesis and its geometric implications}
\label{sec:incl-excl-tangle}

In \cite[Section 3]{Ours1}, we present a detailed trace method, which is the method we follow in
this article. We relate the spectral gap problem to the question of proving that the expansion
$(f_k^\type)_{k \geq \chi(\type)}$ of the volume functions $V_g^\type$ satisfy the \emph{Friedman--Ramanujan
  property}.  We prove that:
\begin{itemize}
\item for any local topological type $\type$ with $\chi(\type) \leq 1$, any $k \geq \chi(\type)$,
 the function $\ell \mapsto \ell f_k^\type(\ell)$ is a Friedman--Ramanujan function in the weak
 sense (see \cite[Theorem 1.13]{Ours1});
\item however, the sum $\sum_{\type : \chi(\type) \leq 1} f_1^\type(\ell)$ is not a
  Friedman--Ramanujan function in the weak sense (see \cite[Section 9]{Ours1}).
\end{itemize}
We prove the second point by exhibiting a subset of surfaces of genus $g$, which are \emph{tangled},
and hence have a small spectral gap. We prove that the probability of this set of surfaces only
decays like $1/g$ as $g \to \infty$. In particular, we observe that
even though these surfaces have vanishing
probability, they contribute large contributions to the averages over the whole moduli space,
due to their exponential proliferation of closed geodesics. This
motivates the need to make a probabilistic hypothesis, removing some ``bad''
surfaces containing many geodesics. Let us introduce a set of ``good surfaces''.

\begin{nota}
  \label{nota:Ninj_Ntang}
  Let $0 < \kappa < 1$. Let us consider the counting functions
  \begin{align*}
    \Ninj(X)
    & = \# \{ \beta \in \mathcal{G}(X)  : \, \ell_X(\beta) \leq \kappa \}
    \\
    \Ntang(X)
    & = \# \{ Y \text{ embedded in } X  : \, \chi(Y) = 1,
      \ell_X^{\mathrm{max}}(\partial Y) \leq \tfL \},
  \end{align*}
  where:
  \begin{itemize}
  \item $\mathcal{G}(X)$ is the set of primitive oriented closed geodesics on $X$, and for $\beta
    \in \mathcal{G}(X)$, $\ell_X(\beta)$ denotes its length;
  \item any embedded subsurface $Y$ is assumed to have geodesic boundary;
  \item $\chi(Y) > 0$ denotes the absolute Euler characteristic of $Y$ (i.e. the absolute value of its
    Euler characteristic);
  \item $\ell_X^{\mathrm{max}}(\partial Y)$ denotes the length of the longest boundary component of
    $Y$;
  \item in $\Ntang$, amongst pair of pants, we only count those with three boundary components
    forming a multi-curve, i.e. we exclude the possibility that two components are glued into a
    once-holed torus (because this once-holed torus is already counted once in $\Ntang$).
  \end{itemize}
  We denote as $\Atf$ the set of hyperbolic surfaces of genus $g$ such
  that $\Ninj = \Ntang = 0$. We call $\kappa$-\emph{short loops} the closed
  geodesics counted by $\Ninj$, and $\tfL$-\emph{tangles} the embedded
  surfaces counted by $\Ntang$.
\end{nota}

\begin{rem}
  The same set of good surfaces appears in \cite{lipnowski2021}, and making the probabilistic
  assumption that $X \in \Atf$ is a crucial step of Lipnowski--Wright's proof of the
  $3/16 - \epsilon$ spectral gap result.
\end{rem}

{
For us, the parameter $\kappa$ will be an arbitrary small fixed positive number, and we will take
$\tfL = \kappa \log(g)$. The probability of $\Atf$ can then be estimated directly using
\cite[Theorem 4.2]{mirzakhani2013} and \cite[Theorem 5]{monk2021a}.
}

\begin{lem}
  \label{lem:prob_TF}
  For any small enough $\kappa > 0$, any large enough $g$, $\tfL = \kappa \log (g)$,
  \begin{equation*}
    1 - \Pwp{\Atf} = \O{\kappa^2 + g^{\frac 3 2 \kappa-1}}.
  \end{equation*}
\end{lem}

We now prove that, under the probability assumption $X \in \Atf$, the number of
geodesics of length $\leq L := A \log(g)$ filling any pair of pants or once-holed torus is radically
reduced, from being exponential in $L$ to polynomial in $L$.

\begin{lem}
  \label{lem:TF_curves}
  Let $0 < \kappa < 1$, $A \geq 1$. For any large enough $g$, if we set $\tfL = \kappa \log(g)$ and
  $L = A \log(g)$, then there exists a set $\mathrm{Loc}_{g}^{\kappa,A}$ of local topological types
  such that:
  \begin{itemize}
  \item for any
  $X \in \Atf$, any primitive closed geodesic of length $\leq L$ filling a surface of absolute Euler
  characteristic~$1$ is locally equivalent to a local type in $\mathrm{Loc}_{g}^{\kappa,A}$;
\item the cardinal of the set satisfies
      \begin{equation}
    \label{eq:bound_number_type_TF}
    \# \mathrm{Loc}_{g}^{\kappa, A} = \O[\kappa, A]{(\log g)^{c_{\kappa,A}}}
  \end{equation}
  for a constant $c_{\kappa, A} > 0$ depending only on $\kappa$ and $A$.
  \end{itemize}
\end{lem}

\begin{proof}
  Let $X \in \Atf$, and let $\gamma$ be a primitive closed geodesic on $X$ of length $\leq L$
  filling a pair of pants $Y$.  Let $b_1, b_2, b_3$ denote the three boundary components of
  $Y$, ordered in non-decreasing length. By definition of $\Atf$,
  \begin{equation*}
    \ell_X(b_1), \ell_X(b_2) \geq \kappa \quad
    \mathrm{and} \quad
    \ell_X(b_3) \geq \tfL.
  \end{equation*}
  Let $\alpha$ denote the simple orthogeodesic of $Y$ from $b_3$ to itself. We fix a base point
  $x_0$ on~$\alpha$, and $a_1$,~$a_2$ two simple paths based at $x_0$, rotating around $b_1$ and
  $b_2$ respectively, so that $\pi_1(Y)$ is the free group generated by $a_1$ and $a_2$
  (see \cref{fig:TF_curves_pop}).

  \begin{figure}[h]
    \centering
    \includegraphics[scale=0.4]{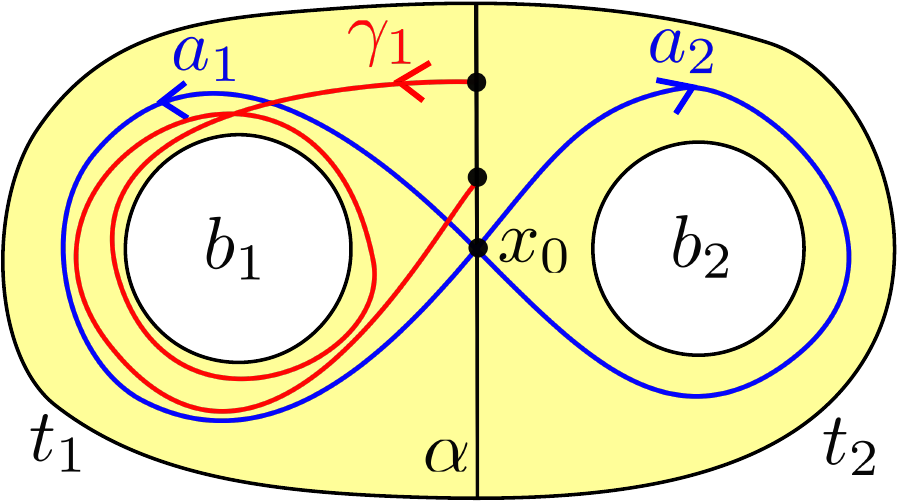}
    \caption{Illustration of the proof of \cref{lem:TF_curves} for pair of
      pants.}
    \label{fig:TF_curves_pop}
  \end{figure}

  The geodesic $\gamma$ fills $Y$, and is in particular not simple. As a consequence, it must
  intersect $\alpha$ (because $Y \setminus \alpha$ is a union of two cylinders). All of its
  intersections are transversal, because $\gamma$ cannot coincide with $\alpha$ on an interval. The
  number of these intersections must be even because $\alpha$ separates $Y$ into two connected
  components. We can therefore write $\gamma$ as the concatenation
  $\gamma_1 \smallbullet \gamma_2 \smallbullet \ldots \smallbullet \gamma_{2k}$, where for every
  $i$, $\gamma_i$ is a geodesic path intersecting $\alpha$ exactly at its endpoints, lying entirely
  on the connected component of $Y \setminus \alpha$ containing $b_{\bar{\imath}}$ for
  $\bar{\imath} \in \{1,2\}$ of the same parity as $i$.

  The path $\alpha$ cuts $b_3$ into two geodesic arcs, of respective lengths $t_1$ and $t_2$. For
  all~$i$, the geodesic path $\gamma_{i}$ is a path from $\alpha$ to itself on the connected
  component $C_{\bar{\imath}}$ of $Y$ containing~$b_{\bar{\imath}}$, that cannot be homotoped to a
  sub-path of $\alpha$ (by minimality of the length of $\gamma$ in its homotopy class). As a
  consequence, its length is greater than the length $t_{\bar{\imath}}$ of the path minimising this
  property. We deduce that
  \begin{equation*}
    L \geq 
    \ell_X(\gamma) =  \sum_{i=1}^{2k} \ell(\gamma_i) \geq k(t_1+t_2)
    \geq k \, \tfL,
  \end{equation*}
  which implies that $k \leq L/\tfL = A/\kappa$.

  For all $i$, $\gamma_i$ is homotopic with endpoints gliding along $\alpha$ to
  $a_{\bar{\imath}}^{m_i}$ for an integer $m_i \in \Z$. The component $C_{\bar{\imath}}$ is
  isometric to part of a hyperbolic cylinder of core of length $\geq \kappa$. Projecting the
  geodesic path $\gamma_i$ on the core $b_{\bar{\imath}}$ following Fermi coordinates decreases its
  length, and in doing so we obtain a path on $b_{\bar{\imath}}$ rotating at least $\abso{m_i}-1$
  times around it. It follows that
  \begin{equation*}
    L \geq \ell(\gamma_i) \geq (\abso{m_i}-1) \kappa,
  \end{equation*}
  which implies that
  $\abso{m_i} \leq M := 1 + \lceil L/\kappa \rceil$.

  To conclude, we have proven that $\gamma$ is freely homotopic to
  $a_{1}^{m_1} \smallbullet a_2^{m_2} \smallbullet \ldots \smallbullet a_1^{m_{2k-1}} \smallbullet
  a_2^{m_{2k}}$ for $k := \lceil A/\kappa \rceil$ and a family of integers $(m_i)_{1 \leq i \leq k}$
  such that $\abso{m_i} \leq M$ for all $i$. We put all of those local topologies in the set
  $\mathrm{Loc}_g^{\kappa,A}$, and note that the number of such configurations is bounded by
  \begin{equation*}
    (2M+1)^{2k} = \O[\kappa,A]{(\log g)^{\lceil 2A/\kappa \rceil}}.
  \end{equation*}

  The proof in the case of a once-holed torus is similar.
\end{proof}

\section{Inclusion-exclusion argument}
\label{sec:inclusion-exclusion}

In order to be able to ``make the assumption $X \in \Atf$'', we need to be able to compute
expectations of functions in which the indicator function
$\1{\Atf} = \1{\Ninj = 0} \times \1{\Ntang = 0}$ appears. A good way to achieve this
is to write this indicator function in terms of counting functions, using an inclusion-exclusion
(such an argument also appears in Friedman's proof of Alon's conjecture~\cite{friedman2003}). The
hope is that these counting functions can then be expressed in terms of geometric functions and
integrated using Mirzakhani's integration formula.

The indicator function $\1{\Ninj = 0}$ can be rewritten in the following way, as
done by Mirzakhani~\cite{mirzakhani2013} and
Lipnowski--Wright~\cite{lipnowski2021}.

\begin{nota}
  For $j \geq 1$, we define
  \begin{equation*}
    \Ninj[,j](X)
    = \# \{ \text{unordered families } \{\beta_1, \ldots, \beta_j\}
    \text{ of distinct } \kappa\text{-short-loops on } X\}.
  \end{equation*}
\end{nota}

\begin{lem}
  \label{lem:inc_excl_inj}
  For any small enough $\kappa>0$, for any $j \geq 1$, all families of geodesics
  counted by $\Ninj[,j]$ are made of simple and disjoint geodesics, and 
  \begin{equation}
    \label{eq:incl_excl_inj}
    \1{\Ninj = 0}
    = 1 - \sum_{j=1}^{+ \infty} (-1)^{j+1} \Ninj[,j]
    = \sum_{j=0}^{+ \infty} (-1)^{j} \Ninj[,j]
  \end{equation}
  with the convention that $\Ninj[,0]=1$.
\end{lem}

\begin{proof}
  We simply observe that, by the binomial theorem,
  \begin{equation*}
    \forall n \geq 0, \quad
    \sum_{j=0}^{n} (-1)^{j} \binom{n}{j}
    = 
    \begin{cases}
      1 & \text{if } n=0 \\
      0 & \text{otherwise.}
    \end{cases}
  \end{equation*}
  The fact that the geodesics are simple and disjoint for small enough $\kappa$
  is a classic application of the Collar Lemma \cite[Theorem 4.1.6]{buser1992}.
\end{proof}

However, there is no reason to believe that all embedded surfaces contributing to $\Ntang$ are
disjoint. For this reason, it is a priori better to treat the indicator function of $\Ntang = 0$
differently, to avoid a tedious enumeration of cases. Luckily, the following easy
inclusion-exclusion is enough for proving \cref{thm:29}.

\begin{nota}
  We define
  \begin{equation*}
    \Ntang[,2]
    := \# \{ \text{unordered pairs } \{Y, Y'\} \text{ of distinct }
    \tfL\text{-tangles on } X \}.
  \end{equation*}
\end{nota}

\begin{lem}
  \label{lem:inc_excl_tang}
  For any $\kappa >0$ and $g \geq 2$, $\tfL = \kappa \log(g)$,
  \begin{equation*}
    \1{\Ntang = 0} = 1 - \Ntang + \O{\Ntang[,2]}.
  \end{equation*}
\end{lem}

\begin{proof}
  The bound is a direct consequence of the inequality
  $ 0 \leq \Ntang - \1{\Ntang \geq 1} \leq \Ntang[,2]$, which is trivially true
  when $X$ contains $0$ or $1$ tangles, and otherwise true because 
  \begin{equation*}
    \forall n \geq 2, \quad 0 \leq n - 1 \leq \frac{n(n-1)}{2} = \binom{n}{2}.
  \end{equation*}
\end{proof}

  \section{Reduction of the number of terms to compute}
\label{s:reduction}

In this subsection we perform some operations which allow us to reduce the number of terms that need
to be computed when performing the inclusion-exclusion.

\subsection{Expectation against $\Ntang[,2]$}
\label{sec:expect-against-ntang}

The following lemma will allow us to discard any expectation in which $\Ntang[,2]$ appears when
proving the $2/9 - \epsilon$ spectral gap result.  This is the reason why the inclusion-exclusion
formula in \cref{lem:inc_excl_tang} will be sufficient for our purposes: the error term will be
negligible. A more careful inclusion-exclusion is necessary to reach the optimal spectral gap
$1/4 - \epsilon$.

\begin{lem}
  \label{lem:incl_exc_work_tang}
  For any $0 < \kappa < 1$, any large enough $g$, $L := 6 \log(g)$, any test function~$F$ supported
  on $[0,L]$, we have
  \begin{equation*}
    \Ewp{\Ntang[,2](X) \sum_{\substack{\gamma \in \mathcal{G}(X) }}
      F(\ell_X(\gamma)) }
    = \O[\kappa]{\frac{\norminf{F(\ell) e^\ell}}{g^{2-19 \kappa}}}.
  \end{equation*}
\end{lem}

\begin{proof}
  First, we notice that, using the triangle inequality, we can consider the case
  $F(x) = e^{-x}\1{[0,L]}(x)$ only.  Let $\gamma$ be a closed geodesic on $X$ of length $\leq L$,
  and $Y$, $Y'$ be two distinct tangles on~$X$. Let $Z := S(\gamma)$ be the surface filled by
  $\gamma$. We define $Z'$ to be the surface obtained by adding any disk in the complement of the
  union $Z \cup Y \cup Y'$. The surface $Z'$ is not necessarily connected, and has at most three
  connected components.  Because the boundary of $Z'$ can be obtained from the boundary of $Z$
  together with $\partial Y$ and $\partial Y'$, we have
  \begin{equation}
    \label{eq:bound_ellZ}
    \ell_X(\partial Z')
    \leq \ell_X(\partial Z) + 6 \tfL
    \leq 2 \ell_X(\gamma) + 6 \kappa \log(g)
    \leq 18 \log(g) = 3L.
  \end{equation}
  
  We observe that the surface $Z'$ contains the surface $Y$ of Euler
  characteristic $-1$. Furthermore, $Z' \neq Y$, because $Z'$ contains $Y'$ and
  $Y \neq Y'$. Since $Z' \setminus Y$ cannot be reduced to one cylinder (by
  definition of the notion of tangle), $r := \chi(Z') \geq 2$.

  As in the proof of \cite[Theorem 6.1]{Ours1}, we use events of extremely small probability to
  bound the Euler characteristic of $Z'$ and reduce the sum to a finite number of terms. Indeed, for
  any $X \in \cM_g$, by the naive uniform bound on the number of closed geodesics under a specific
  length, see e.g. \cite[Lemma 2.2]{Ours1},
    \begin{equation}
      \label{eq:supnormt2}
      \Ntang[,2](X) \sum_{\substack{\gamma \in \mathcal{G}(X) \\ \ell_{X}(\gamma) \leq L}}
      e^{-\ell_X(\gamma)}  
      = \O{(g e^\tfL)^6 g L } = \O{\log(g) g^{13}} 
    \end{equation}
    because any pair of tangles is determined by at most~$6$ geodesics of lengths $\leq \tfL$.  We
    adapt the proof of \cite[Proposition 6.4]{Ours1} to prove that the probability for a surface of
    genus $g$ to contain a surface $Z'$ such as above or Euler characteristic $> \chi$ is
    $$\O[\chi]{\frac{(L+\tfL)^{c(\chi)}e^{L+3 \tfL}}{g^{\chi+1}}}
    = \O[\chi]{\frac{(\log(g))^{c(\chi)}}{g^{\chi-8}}}.$$ It
    follows, taking $\chi=24$, that
    \begin{equation*}
      \Ewp{\Ntang[,2](X) \1{[25,\infty)}(\chi(Z'))
        \sum_{\substack{\gamma \in \mathcal{G}(X) \\ \ell_{X}(\gamma) \leq L}}
      e^{-\ell_X(\gamma)} }
    = \O{\frac{(\log(g))^{c(24)+1}}{g^{24-21}} } = \O{\frac{1}{g^{2}}}.
  \end{equation*}
  We are therefore left to study the sum
  \begin{equation}
    \label{eq:Ntang2_TF_gene_2}
    \sum_{r=2}^{24} 
    \Ewp{\sum_{\substack{(\gamma, Y, Y') \\ \ell_{X}(\gamma) \leq L \\ \chi(Z')=r}}
      e^{-\ell_X(\gamma)}}
  \end{equation}
  where the sum runs over all families $(\gamma, Y, Y')$ such that $\gamma$ is a primitive closed
  geodesic on $X$ of length $\leq L =  6 \log(g)$, $Y$ and $Y'$ are disjoint tangles on $X$, and
  $\chi(Z') = r$.

  Let us fix a filling type $\Sf'$ of absolute Euler characteristic $2 \leq r \leq 24$, with at
  most three connected components, and a connected subsurface $\Sf \subseteq \Sf'$ (we allow that
  $\Sf$ is a connected component of $\Sf'$, or shares some of its boundary components). By
  \eqref{eq:bound_ellZ}, the quantity we need to bound is smaller than
  \begin{equation}
    \label{eq:proof_Ntang_2_enumeration}
    \Ewp{\sum_{\substack{(Z,Z')  \\ \text{ homeo to } (\Sf,\Sf') \\
          \ell_X(\partial Z), \ell_X(\partial Z') \leq 3L}}
      \Ntang[,2](Z') 
      \sum_{\substack{\gamma \text{ filling } Z \\ \ell_X(\gamma) \leq L}} e^{-
        \ell_X(\gamma)}
      \1{\ell_X(\partial Z') \leq \ell_X(\partial Z)+6 \tfL}}.
  \end{equation}
  We once again use \cite[Lemma 2.2]{Ours1} to bound the number of pairs of tangles, and obtain that
  for any $Z'$ as above, $\Ntang[,2](Z') = \O{(r e^{\tfL})^6} = \O{g^{6 \kappa}}$.  We use Wu and
  Xue's counting result, see \cite[Theorem 2.4]{Ours1}, with $\eta = \kappa$ to bound the number of
  possibilities for the loop $\gamma$ filling $Z$, and obtain that
  \begin{align*}
    \sum_{\substack{\gamma \text{ filling } Z \\ \ell_X(\gamma) \leq L}}
    e^{- \ell_X(\gamma)} 
    = \O[\kappa]{L
    \exp \Big(- \frac{1-\kappa}{2} \ell_X(\partial Z)\Big)}
    = \O[\kappa]{g^{12\kappa} \exp \Big(- \frac{\ell_X(\partial Z')}{2}\Big)}
  \end{align*}
  because, by the hypotheses on the lengths of $\partial Z$ and $\partial Z'$ in
  \eqref{eq:proof_Ntang_2_enumeration},
  \begin{equation*}
    \frac{1 - \kappa}{2} \ell_X(\partial Z)
    \geq \frac{1 - \kappa}{2} \ell_X(\partial Z')
    - 3 \kappa (1-\kappa) \log(g)
    > \frac{\ell_X(\partial Z')}{2} - 12 \kappa \log(g).
  \end{equation*}
  Therefore, 
  \begin{equation*}
    \eqref{eq:proof_Ntang_2_enumeration} = \mathcal{O}_\kappa \Bigg(
      g^{18 \kappa} \Ewp{\sum_{\substack{(Z,Z')  \\ \text{ homeo to
          } (\Sf,\Sf') \\ \ell_X(\partial Z), \ell_X(\partial Z') \leq 3L}}
      \exp \paren*{- \frac{\ell_X(\partial Z')}{2}}}\Bigg).
  \end{equation*}
  We use Mirzakhani's integration formula to compute this expectation, and
  obtain that we need to bound
  \begin{equation}
    \label{eq:proof_Ntang_2_integrate}
    g^{18 \kappa}
    \int_{\substack{\norm{\x}_1, \norm{\y}_1 \leq 3L }}
    \phi_g^{\Sf'}(\x) V_{\Sf}(\y) V_{\Sf' \setminus \Sf}(\x,\y)
    \exp \paren*{- \frac{\norm{\x}_1}{2}} 
    \d \x \d \y
  \end{equation}
  where
  \begin{itemize}
  \item $\x$ denotes the length-vector of $\partial Z'$ and $\y$ that of $\partial Z$;
  \item $\norm{\cdot}_1$ is the $\ell^1$-norm on $\R^n$ for $n \geq 1$;
  \item $\phi_g^{\Sf'}$ is sum over all realizations of $\Sf'$ in $S_g$ as introduced in
    \cite[Theorem~5.7]{Ours1} for connected filling types, and generalised in
    \cite[Section 5.5]{Ours1} to more general filling types;
  \item $V_{\Sf'}(\x)$ and $V_{\Sf' \setminus \Sf}(\x,\y)$ respectively stand for the product of the
    Weil--Petersson volumes of the components of $\Sf'$ and $\Sf' \setminus \Sf$.
  \end{itemize}

  We bound the Weil--Petersson volume $V_{\Sf' \setminus \Sf}(\x,\y)$ by the naive polynomial
  bound (see \cite[equation (2.8)]{Ours1}), and note that its values for $\x = \mathbf{0}$ and
  $\y = \mathbf{0}$ are $\O{1}$ because $r \leq 24$. We naturally extend the bound on
  $\phi_g^{\Sf'}$ from \cite[Proposition 5.22]{Ours1} with $\ord = r $ to the case of disconnected
  filling types. Altogether, this allows us to deduce that \eqref{eq:proof_Ntang_2_integrate} is
  $\O{g^{18 \kappa - r} (3L)^{c(r)}}$ for a constant $c(r)$ depending only on~$r$.
  Summing these inequalities for $2 \leq r \leq 24$, we obtain that the expectation we study is
  $\O[\kappa]{g^{19 \kappa -2}}$, which is our claim.
\end{proof}

{
\subsection{Bounding the number of components separated by the $\kappa$-short loops}

When analyzing the infinite sum $\sum_j (-1)^j \Ninj[,j]$, we will have to worry about the
dependency of the bounds we use on the integer $j$, the number of $\kappa$-short-loops. In
order to do so, we exclude another event of extremely small probability to bound the number of
connected components separated by the $\kappa$-short loops using \cref{s:A3}.
More precisely, for an integer $Q$, we define in \cref{s:A3} the set $\MC_X(Q)$ of multi-curves
separating $X$ into at most $Q$ connected components. The event $\cB_g^{ \kappa, Q}$ then groups
genus $g$ hyperbolic surfaces such that any $\kappa$-short multi-curve belongs in $\MC_X(Q)$; its
probability is estimated in \cref{l:PBN}.  Clearly $\Ninj(X)=0$ implies $X \in \cB_g^{ \kappa, Q}$
and hence $\1{\Ninj=0} = \1{\cB_g^{ \kappa, Q}}\1{\Ninj=0}$. As a consequence,
\cref{eq:incl_excl_inj} can be rewritten as
  \begin{equation*}
    \1{\Ninj = 0}
    = \1{\cB_g^{ \kappa, Q}} \sum_{j=0}^\infty (-1)^j \Ninj[,j] 
    = \1{\cB_g^{ \kappa, Q}} \sum_{j=0}^\infty (-1)^j \Ninj[,j,Q] 
  \end{equation*}
  where $\Ninj[,j,Q](X)$ now counts unordered $\kappa$-short multi-loops in $\MC_X(Q)$ (see
  \eqref{eq:NinjQ}). We then write:
  \begin{equation}
    \label{e:Ninj_rewrite_Q}
    \1{\Ninj = 0}
    = \sum_{j=0}^\infty (-1)^j \Ninj[,j,Q]
    + (1-\1{\cB_g^{ \kappa, Q}}) \sum_{j=0}^\infty (-1)^j \Ninj[,j,Q].
  \end{equation}

We prove the following, which will allow us to replace $\1{\Ninj=0}$ by the first sum above in the
proof of the $2/9$ spectral gap result.

\begin{lem}
  \label{lem:Q_app}
  For any $0 < \kappa < 1$, any large enough $g$, $L:=6 \log(g)$, any test function $F$ supported on
  $[0,L]$, if we pick $Q:=77$,
  \begin{equation*}
    \Ewp{(1+\Ntang)(1-\1{\cB_g^{ \kappa, Q}}) \sum_{j=0}^\infty  \Ninj[,j,Q]
    \sum_{\gamma \in \mathcal{G}(X)} |F(\ell_X(\gamma))|} = \O{\|F\|_\infty}.
  \end{equation*}
\end{lem}

\begin{proof}
  By Cauchy--Schwarz, the expectation above can be bounded by the product
  $\Pwpo(X\notin \cB_g^{ \kappa, Q})^{1/2} E^{1/2}$ where
  \begin{equation*}
    E := 
    \Ewp{\Big((1+\Ntang) \sum_{j=0}^\infty \Ninj[,j,Q] \sum_{\gamma \in \mathcal{G}(X)}
      |F(\ell_X(\gamma))| \Big)^2}.
  \end{equation*}
  By \cref{l:PBN}, the probability satisfies
  \begin{equation*}
    \Pwp{X\notin \cB_g^{ \kappa, Q}}^{1/2} = \O[\kappa,Q]{\frac{1}{g^{(Q-1)/2}}}.
  \end{equation*}
  The value of $Q$ is picked to compensate the square root of the expectation. Indeed, by the
  uniform geodesic counting argument, \cite[Lemma 2.2]{Ours1},
  \begin{equation*}
    E = \O{\|F\|^2_\infty \big(L e^L g (g e^\tfL)^6\big)^2 \, \Ewpo[g] [(\cY_{\kappa,Q}+1)^2]}
  \end{equation*}
  where $\cY_{\kappa,Q} =  \sum_{j=1}^\infty \Ninj[,j,Q] $ is introduced in \cref{s:A3}. We prove
  in \cref{p:YN} that the expectation above is $\O[\kappa,Q]{1}$ and hence
  $$E = \O[\kappa,Q]{\|F\|_\infty^2 (\log(g))^2 g^{2(6+1+6(1+\kappa))}}
  = \O[\kappa,Q]{\|F\|_\infty^2 g^{38}}.$$
  Taking $Q = 2 \times 38+1=77$ yields the claimed result.
\end{proof}

Let us now prove that we can truncate the first sum in \eqref{e:Ninj_rewrite_Q} to only account for
integers $j \leq \log g$.

\begin{lem}
  \label{lem:bound_j_inc}
  With the notations of \cref{lem:Q_app},
  \begin{equation*}
    \Ewp{(1+\Ntang) \sum_{j > \log g} \Ninj[,j,Q]
    \sum_{\gamma \in \mathcal{G}(X)} |F(\ell_X(\gamma))|} = \O{\|F\|_\infty}.
  \end{equation*}
\end{lem}

\begin{proof}
  This is a direct consequence of the tail estimate \cref{lem:tail_j} together with the trivial uniform
  counting bounds on $\Ntang$ and the sum over $\cG(X)$ coming from the uniform bound \cite[Lemma 2.2]{Ours1}.
\end{proof}

We shall now introduce some notation to accommodate to the counting functions~$\Ninj[, j,Q]$.

\begin{nota}
  Let $\type$ be a local type. For an integer $j$, we define as $\rho_j \type$ the local type
  obtained by adding $j$ copies of the local type simple to $\type$, i.e.
  $$\rho_j \type = (\type, \underbrace{\mathbf{s}, \ldots, \mathbf{s}}_{j \text{ times}}).$$
  Similarly, for a filling type $\Sf$, we define $\rho_j\Sf$ to be the filling type obtained by
  adding $j$ copies of the cylinder to $\Sf$.
\end{nota}

\begin{nota}
  For integers $j$, $Q$, test functions $F : \R_{>0} \rightarrow \C$ and
  $\mu : \R_{>0}^j \rightarrow \C$, we define the $Q$-bounded average
  \begin{equation*}
    \avQ[\rho_j \type]{\mu \otimes F}
    = \Ewp{\sum_{\substack{\gamma \in \mathcal{G}(X) \\ \beta \in \MC_X(Q) \\ (\gamma,\beta) \sim \rho^j\type}}
      F(\ell_X(\gamma)) \mu(\ell_X(\beta))}.
  \end{equation*}
\end{nota}
Following the conventions of \cite[Section 5.5]{Ours1}, we also define the averages
$$\av[\rho_j \vec{\type}]{\mu \otimes (F_1, \ldots, F_m)}$$ for a family of types $\vec{\type}$ and of
test functions $F_1, \ldots, F_m$.

\begin{rem}
  In the following, we will be particularly interested in the case of the function
  \begin{equation}
    \label{eq:mukappa}
    \mu_\kappa^j : (x_1, \ldots, x_j) \mapsto \frac{(-1)^j}{j!} \prod_{i=1}^j \1{[0,\kappa]}(x_i)
  \end{equation}
  which appears naturally in the inclusion-exclusion.
\end{rem}

We can then extend the integration formula, \cite[Theorem 5.7]{Ours1}, to the $Q$-bounded averages.
\begin{lem}
  \label{lem:express_average_j}
  For any local type $\type = \eqc{\Sf,\curve}$, any integers $j \geq 1$, $Q \geq 0$, $g \geq 3$,
  any test functions $F:\R_{>0} \rightarrow \C$ and $\mu:\R^j_{>0} \rightarrow \C$,
  \begin{equation}
    \avQ[\rho_j \type]{\mu \otimes F}
    = \frac{1}{n(\type)} \int_{\mathbf{y} \in \R_{>0}^j} \int_{\cT_{g_\Sf,n_\Sf}^*}
    F(\ell_Y(\gamma))   \,  \mu(\mathbf{y}) \,
    \phi_{g,Q}^{\rho_j \Sf}(\x,\mathbf{y})
    \d \Volwp[g_\Sf,n_\Sf](\x,Y) \d \mathbf{y}
  \end{equation}
  where
  \begin{equation}
    \label{eq:phi_ext_Q}
     \phi_{g,Q}^{\rho_j \Sf}(\x, \mathbf{y})
    := \frac{x_1 \ldots x_{n_\Sf} y_1 \ldots y_j}{V_g} 
    \sum_{\mathfrak{R} \in R_{g,Q}(\rho_j\Sf)} V_{\mathfrak{R}}(\x,\mathbf{y},\mathbf{y})
  \end{equation}
  and $R_{g,Q}(\rho_j\Sf)$ is the set of realizations of $\Sf$ and $j$ cylinders in a surface of
  genus $g$ such that the $j$ cylinders separate it in at most $Q$ connected
  components.
\end{lem}

In the general case (without $j$ and $Q$), we would use \cite[Lemma 5.21]{Ours1} to reduce the number
of terms in \eqref{eq:phi_ext_Q} using the notion of rank of a realization. This result needs to be
adapted for the inclusion-exclusion to account for the dependency in~$j$; this is done in
\cref{lem:limit_rank_Q}. }

\section{Expression of the terms arising in the inclusion-exclusion}
\label{sec:reduct-numb-topol}

The aim of this section is to express sums of the form
\begin{equation}
  \label{eq:av_cond}
   \av[\type]{F \, | \, X \in \Atf} :=
  \Ewp{\sum_{\gamma \sim \type} F(\ell_X(\gamma)) \, \1{\Atf}(X)} 
\end{equation}
for a local type $\type$ under a manageable form building on the results above, up to errors which
will be negligible for our purposes.
The key starting point is to rewrite
\begin{equation}
  \label{eq:atf_app} \1{\Atf}
  = (1-\Ntang) \sum_{j=0}^{\lfloor \log g \rfloor} (-1)^j \Ninj[,j,Q]  + \text{ error}
\end{equation}
where the error term corresponds to the contributions estimated in \cref{s:reduction}.

\subsection{The local type simple}

We prove the following explicit inclusion-exclusion formula for the local type simple.

\begin{nota}
  Let us introduce a few notations useful to the following statement.
  \begin{itemize}
  \item $\type_{0,3}^\partial$ and $\type_{1,1}^\partial$ are the local types $\eqc{\Sf, \partial \Sf}$ with
    $\Sf = \Sf_{0,3}$ and $\Sf_{1,1}$ respectively;
  \item $\type_{1,1}^{\mathrm{s},\partial} =\eqc{\Sf_{1,1}, (\beta, \partial \Sf_{1,1})}$ where
    $\beta$ is the essential simple closed loop on the once-holed torus $\Sf_{1,1}$;
  \item $\mathrm{Loc}_{1,1}^{2 \mathrm{s}}$ is the set of local topological types of pairs of simple
    loops filling $\Sf_{1,1}$;
  \item $\mathrm{Loc}_{1,1}^{2 \mathrm{s},\partial}$ is the set of local types
    $\eqc{\Sf_{1,1},(\beta_1, \beta_2, \partial \Sf_{1,1})}$ where $(\beta_1, \beta_2)$ are two
    simple loops on $\Sf_{1,1}$.
  \end{itemize}
\end{nota}

\begin{figure}[h!]
  \centering
  \includegraphics[scale=0.7]{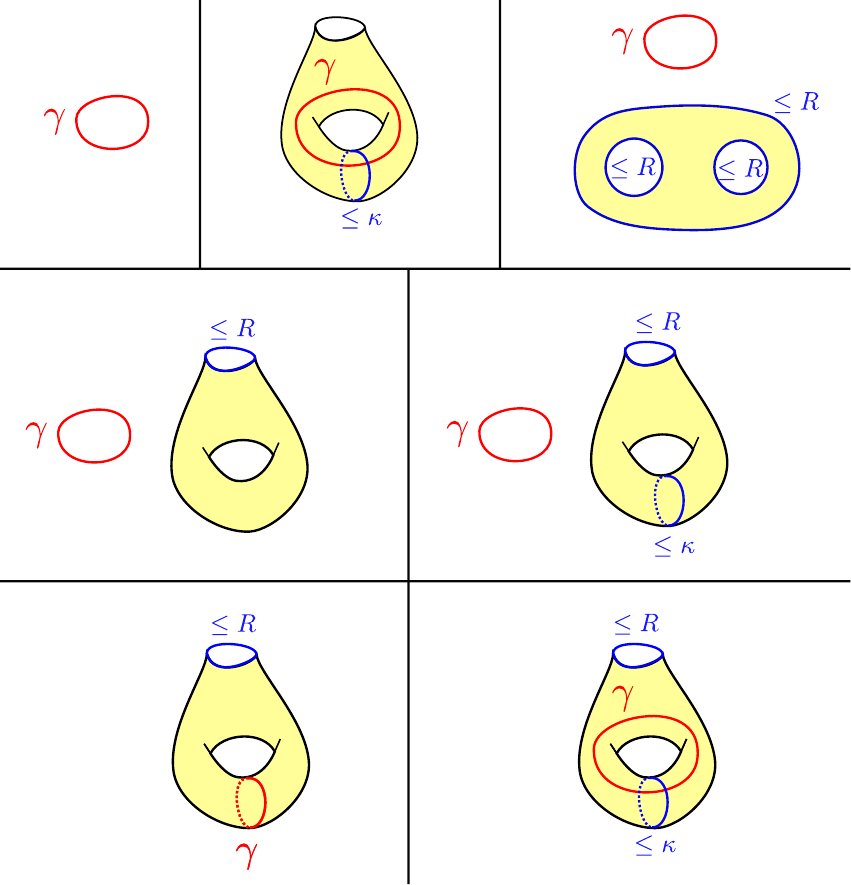}
  \caption{Illustration of the different terms in \cref{p:incl_s}.}
  \label{fig:incl_0}
\end{figure}

\begin{prp}
  \label{p:incl_s}
  There exists a constant $c>0$ such that, for any small enough $\kappa>0$, any large enough $g$,
  $L = 6 \log(g)$, $\tfL = \kappa \log(g)$, $Q=77$, any test function $F$ supported on $[0,L]$, the
  average $ \av[\mathbf{s}]{F \, | \, X \in \Atf} $ is equal to 
  \begin{align}
    \label{e:exp_TF_1}
    \sum_{j=0}^{\lfloor \log g \rfloor} \Big[
    & \avQ[\rho_j\mathbf{s}]{\mu_\kappa^j \otimes F}
      - \sum_{\type \in \mathrm{Loc}_{1,1}^{2 \mathrm{s}}}
      \avQ[\rho_j \type]{\mu_\kappa^j \otimes (F,\mu_\kappa^1)}
    - \avQ[\rho_j(\mathbf{s},\type_{0,3}^\partial)]{\mu_\kappa^j \otimes
      (F,\mu_{\tfL}^{3})} \\
    \label{e:exp_TF_2}
    &  - \avQ[\rho_j(\mathbf{s},\type_{1,1}^\partial)]{\mu_\kappa^j \otimes
      (F,\mu_{\tfL}^{1})}
      + \avQ[\rho_j(\mathbf{s},\type_{1,1}^{\mathrm{s},\partial})]{\mu_\kappa^j \otimes
      (F,(\mu_\kappa^1,\mu_\tfL^1))}
    \\
    & \label{e:exp_TF_3}
     - \avQ[\rho_j\type_{1,1}^{\mathrm{s},\partial}]
      {\mu_\kappa^j \otimes (F,\mu_\tfL^1)}
    + 
      \sum_{\type \in \mathrm{Loc}_{1,1}^{2 \mathrm{s}, \partial}}
      \avQ[\rho_j \type]{\mu_\kappa^j \otimes
      (F,\mu_\kappa^1,\mu_\tfL^1)}   \Big]
  \end{align}
  up to an error of size $\O{\|F\|_\infty + \|F(\ell) e^{\ell}\|_\infty/g^{2-c \kappa}}$.
\end{prp}

The terms in this formula correspond to the different topologies one needs to enumerate for the
simple closed geodesic $\gamma$, the $\kappa$-short loops and the $\tfL$-tangle. They are
represented in \cref{fig:incl_0}, in the same arrangement as the formula.

\begin{proof}
  First, we observe that by the triangle inequality,
  \begin{align*}
    \av{F \1{[0,\kappa]}}
     \leq \norminf{F} \, \Ewpo \brac*{\# \{\gamma \in \mathcal{G}(X) \, : \, \ell_X(\gamma) \leq
        \kappa\}} 
     = \O{\kappa^2 \norminf{F}} 
  \end{align*}
  by \cite[Section 4.2]{mirzakhani2013}.  We can therefore restrict the sum we wish to study to
  geodesics~$\gamma$ of length $> \kappa$. By \cref{lem:inc_excl_inj}, for small enough $\kappa$,
  all geodesics of length $\leq \kappa$ are disjoint.

  As suggested by the beginning of this section, we substitute $\1{\Atf}$ in \eqref{eq:av_cond}
  using \eqref{eq:atf_app}. The error is bounded in Lemmas \ref{lem:incl_exc_work_tang} to
  \ref{lem:bound_j_inc} by $\norminf{F(\ell)e^\ell}/g^{2-19 \kappa} +\norminf{F}$. We
  are therefore left with expressing the sums
  \begin{equation}
    \label{eq:incl_exc_imp_1}
    \sum_{j=0}^{\lfloor \log g \rfloor} (-1)^j
    \Ewp{\Ninj[,j,Q](X)\sum_{\gamma \text{ simple}} F(\ell_X(\gamma))}
  \end{equation}
  and
  \begin{equation}
    \label{eq:incl_exc_imp_2}
     \sum_{j=0}^{\lfloor \log g \rfloor} (-1)^{j+1}
    \Ewp{\Ninj[,j,Q](X) \Ntang(X)\sum_{\gamma \text{ simple}} F(\ell_X(\gamma))}.
  \end{equation}
  We only detail the computation for the first average, the second following the same lines.

  We pick a numbering of the $j$-tuples counted in $\Ninj[,j,Q]$, which allows us to rewrite the
  sum as 
  \begin{equation}
    \label{eq:proof_TF_ind_1}
    \sum_{j=0}^{\lfloor \log g \rfloor} \Ewp{
            \sum_{\substack{\beta \in \MC_X(Q)}} \mu_\kappa^j(\ell_X(\beta))
      \sum_{\substack{\gamma \text{ simple} \\ \kappa < \ell_X(\gamma) \leq L}} F(\ell_X(\gamma))
       }
  \end{equation}
  for the function $\mu_\kappa^j$ introduced in \eqref{eq:mukappa}.
  We notice that the hypothesis $\ell_X(\gamma) > \kappa$ implies that $\gamma$ is distinct from
  $\beta_i^{\pm 1}$ for all~$i$.  As in the proof of \cref{lem:incl_exc_work_tang}, we regroup the
  terms of this sum according to the topology $\Sf'$ of the surface filled by the multi-loop
  $(\gamma, \beta_1, \ldots, \beta_j)$. This allows us to prove, using the same strategy, that there
  exists a constant $c>0$ such that the restriction of \cref{eq:proof_TF_ind_1} to multi-loops such
  that $\chi(\Sf') \geq 2$ is $\O[\kappa]{\norminf{F(\ell) \, e^\ell}/g^{2-c \kappa}}$.
  Hence, we only need to enumerate local topologies for which $\chi(\Sf') \leq 1$.

  The only local topology for which $\chi(\Sf')=0$ is the product $\mathbf{s}^{j+1}$ of $j+1$ copies
  of the local type ``simple''. By definition, this local topology corresponds to the average
  $\av[\rho_j\mathbf{s}]{\mu_\kappa^j \otimes F}$ appearing as the first term of \eqref{e:exp_TF_1}.

  Now, we carefully enumerate the local topologies for which $\chi(\Sf') = 1$. We recall that, by
  hypothesis on $\kappa$, for any multi-loop $(\gamma, \beta_1, \ldots, \beta_j)$ contributing to
  the sum, all the $\beta_j$ are simple and disjoint. Furthermore, $\gamma$ is simple, and distinct
  from any $\beta_i^{\pm 1}$ (but not necessarily disjoint).  The only way to obtain a local type of
  absolute Euler characteristic $1$ then is to take a type $\type \times \mathbf{s}^{j-k}$, where
  $\type = \eqc{\Sf_{1,1},(\gamma, \beta_{i_1}, \ldots, \beta_{i_k})}$ or
  $\eqc{\Sf_{0,3},(\gamma, \beta_{i_1}, \ldots, \beta_{i_k})}$ for a $k \leq j$ and a multi-loop
  $(\gamma, \beta_{i_1}, \ldots, \beta_{i_k})$ filling $\Sf_{1,1}$ or $\Sf_{0,3}$.
  \begin{itemize}
  \item Any simple loop on the pair of pants $\Sf_{0,3}$ is homotopic to one of its boundary
    components. It is therefore impossible to fill a pair of pants with a multi-loop containing only
    simple loops, so the case $\Sf_{0,3}$ cannot happen.
  \item If $\type = \eqc{\Sf_{1,1},(\gamma, \beta_{i_1}, \ldots, \beta_{i_k})}$, in order for the
    multi-loop to fill the once-holed torus, we need $k \geq 1$, because $\gamma$ is simple. But a
    once-holed torus cannot contain two disjoint essential curves, so $k=1$. After the change of
    variables $j \rightarrow j-1$, counting the number of possibilities for the choice of the index
    $i_1$, this yields the second term of \eqref{e:exp_TF_1}.
  \end{itemize}
  This allows us to conclude for \eqref{eq:incl_exc_imp_1}.

  For \eqref{eq:incl_exc_imp_2}, we now enumerate possibilities for the simple loop $\gamma$, the
  multi-curve $(\beta_1, \ldots, \beta_j)$ and the tangle $\Sf_{0,3}$ or $\Sf_{1,1}$. Note that
  possibilities are limited because the absolute Euler characteristic of the tangle is $1$ and we
  can exclude all situations for which the total absolute Euler characteristic is $>1$. We therefore
  obtain the following possibilities.
  \begin{itemize}
  \item If none of the loops $\gamma, \beta_1, \ldots, \beta_j$ is contained inside the tangle, then
    because the total Euler characteristic is at most $1$, this family of loops must have no
    self-intersection. Depending on whether the tangle is a pair of pants or a once-holed torus, we
    obtain the third term of \eqref{e:exp_TF_1} and the first term of \eqref{e:exp_TF_2}.
  \item If one of the loops is included in the tangle, then the tangle must be a once-holed
    torus. Then, since a once-holed torus cannot contain more that one simple disjoint loop, we are
    reduced to the following possibilities:
    \begin{itemize}
    \item either the tangle contains exactly one $\beta_k$ in its interior, which corresponds to
      the second term in \eqref{e:exp_TF_2};
    \item or the tangle contains $\gamma$ and no $\beta_k$ in its interior, which corresponds to the
      first term in \eqref{e:exp_TF_3};
    \item or it contains $\gamma$ and one $\beta_k$, which is the second term in \eqref{e:exp_TF_3}.
    \end{itemize}
  \end{itemize}
\end{proof}

\subsection{Local types filling a pair of pants}

Let us now state the equivalent of \cref{p:incl_s} for a local type filling a pair of pants. The
enumeration of cases is simpler here as the surface $S(\gamma)$ already has absolute Euler
characteristic $1$.

\begin{nota}
  Let $\type = \eqc{\Sf_{0,3},\curve}$ be a local type filling a pair of pants. We write
  $\type^\partial := \eqc{\Sf_{0,3},(\curve, \partial \Sf_{0,3})}$. We denote as
  $\mathrm{Loc}_{1,1}^{\type,\mathrm{s},\partial}$ the set of local types
  $\eqc{\Sf_{1,1},(\beta_1,\beta_2,\partial \Sf_{1,1})}$ where $\beta_1$ is a loop of local type
  $\type$ on $\Sf_{1,1}$ and $\beta_2$ is a simple loop on $\Sf_{1,1}$ so that the pair
  $(\beta_1,\beta_2)$ fills~$\Sf_{1,1}$.
\end{nota}

\begin{prp}
  \label{p:incl_pop}
  There exists a constant $c>0$ such that, for any small enough $\kappa>0$, any large enough $g$,
  $L = 6 \log(g)$, $\tfL = \kappa \log(g)$, $Q=77$, any test function $F$ supported on $[0,L]$, any
  local type $\type$ of filling type $(0,3)$, the average $ \av[\type]{F \, | \, X \in \Atf} $ is
  equal to
  \begin{align}
    \label{e:exp_TF_1_pop}
    \sum_{j=0}^{\lfloor \log g \rfloor} \Big[
    \avQ[\rho_j\type]{\mu_\kappa^j \otimes F}
    - \avQ[\rho_j\type^\partial]{\mu_\kappa^j \otimes (F,\mu_\tfL^3)}
    + \sum_{\type' \in \mathrm{Loc}_{1,1}^{\type,\mathrm{s},\partial}}
     \avQ[\rho_j\type']{\mu_\kappa^j \otimes (F,\mu_\kappa^1,\mu_\tfL^1)}
    \Big]
  \end{align}
  up to an error $
  \O{\|F\|_\infty + \|F(\ell) e^{\ell}\|_\infty/g^{2-c \kappa}}$.
\end{prp}

\begin{figure}[h!]
  \centering
  \includegraphics[scale=0.8]{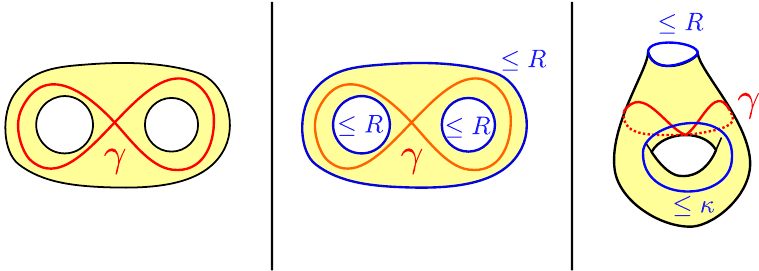}
  \caption{Illustration of the terms of \cref{p:incl_pop}.}
  \label{fig:incl_1}
\end{figure}

\subsection{Writing of these expectations as Friedman--Ramanujan functions}
\label{sec:proof-refch-these}

We are now able to prove the following proposition. This will allow us to apply the properties of
Friedman--Ramanujan functions, and in particular the cancellation argument presented in
\cite[Section 3.4]{Ours1}, in the context of the inclusion-exclusion.

\begin{prp}
  \label{prp:FR_incl_exclu}
  There exists a constant $c >0$ satisfying the following. For any small enough $\kappa >0$, any
  large enough $g$, $L = 6 \log(g)$, $\tfL=\kappa \log(g)$, any local topology~$\type$ filling a
  cylinder or a pair of pants, there exists a density function
  $A_{\type,g}^{\kappa} : \R_{> 0} \rightarrow \R$ satisfying the following. For any test
  function~$F$ of support included in $[0,L]$,
  \begin{align*}
    \av[\type]{F \, | \, X \in \Atf} =  \int_{0}^{+\infty} F(\ell) \,
      A_{\type,g}^{\kappa}(\ell) \d \ell
      + \O[\kappa]{\frac{\norminf{F(\ell) \, e^\ell}}{g^{2-c\kappa}} + g^{c \kappa} \norminf{F(\ell)
      e^{\frac \ell 2}}}.
  \end{align*}
  Furthermore, the function $\ell \mapsto \ell A_{\type,g}^{\kappa}(\ell)$ belongs in the weak
  Friedman--Ramanujan class~$\cF_w^{c,c}$, and its $\cF_w^{c,c}$-norm is $\O[\kappa]{g^{c \kappa}}$.
\end{prp}

\begin{proof}
  Let $\type = \eqc{\mathbf{P},\curve}$ for a loop $\curve$ filling the pair of pants
  $\mathbf{P}$. The length of $\curve$ in terms of the boundary lengths of $\mathbf{P}$ is given by
  a function $h : \R_{>0}^3 \rightarrow \R_{>0}$.  For $j \geq 0$, by \cref{lem:express_average_j},
  the average $\avQ[\rho_j \type]{ \mu_\kappa^j \otimes F}$ present in \eqref{e:exp_TF_1_pop} can be
  written as
  \begin{equation*}
    \frac{(-1)^j}{j!}
    \frac{1}{n(\type)}
    \int_{\R_{>0}^3} \int_{[0, \kappa]^{j}}
    F(h(\ell_1, \ell_2, \ell_3)) \,
    \phi_{g,Q}^{\rho_j \mathbf{P}}(\ell_1, \ell_2, \ell_3, \x, \x) 
     \d \ell_1 \d \ell_2 \d \ell_3 \d \x.
   \end{equation*}
   The filling type $\rho_j \mathbf{P}$ has absolute Euler characteristic $1$ and we will
     hence only need to compute the integral above at the leading order to estimate it up to errors
     decaying in $\O[\kappa]{\norminf{F(\ell)e^\ell}/g^{2-\kappa}}$. We first use
     \cref{lem:limit_rank_Q} to replace $\phi_{g,Q}^{\rho_j \mathbf{P}}$ by a sum over a fixed set
     of $\O[Q]{1}$ realizations of rank $1$. Let us now compute the contribution of these
   realizations of rank $1$. First, we compute the term of the connected realization, which is the
   case when we glue a surface of signature $(g-2-j,2j+3)$ to the pair of pants $\mathbf{P}$ and the
   $j$ cylinders. This yields a density function
  \begin{equation*}
    \frac{(-1)^j}{j!} \frac{1}{n(\type)}
    \int_{\star} \int_{[0, \kappa]^{j}} 
    \frac{V_{g-2-j,2j+3}(\ell_1, \ell_2, \ell_3, \x,\x)}{V_g}
    \ell_1 \ell_2 \ell_3
    \frac{\d \ell_1 \d \ell_2 \d \ell_3}{\d \ell}
    \prod_{i=1}^j x_i \d x_i
  \end{equation*}
  where the integral $\int_\star$ runs on the level-set
  $\{(\ell_1,\ell_2,\ell_3) : h(\ell_1, \ell_2, \ell_3) = \ell\}$. By
    \cref{theo:volume_asympt_exp}, we can reduce the problem to computing
  \begin{align*}
    & \frac{(-1)^j}{j!} \frac{1}{n(\type)} \frac{V_{g-2-j,2j+3}}{V_g} 
    \int_{\star}
    \int_{[0, \kappa]^{j}}
    \prod_{i=1}^3 2 \sinh \div{\ell_i} \prod_{i=1}^j \frac{4}{x_i} \sinh^2 \div{x_i}
       \frac{\d \ell_1 \d \ell_2 \d \ell_3}{\d \ell} \d \x \\
    & = \frac{(-1)^j I(\kappa)^j}{j!} \frac{1}{n(\type)} \frac{V_{g-2-j,2j+3}}{V_g}
      \int_{\star}
      \prod_{i=1}^3 f_1(\ell_i) 
      \frac{\d \ell_1 \d \ell_2 \d \ell_3}{\d \ell}
  \end{align*}
  where $f_1(x) = 2 \sinh \div{x}$ and
  $I(\kappa) := \int_{0}^\kappa 4\sinh^2 \div{x} \d x/x = \O{\kappa}$ for $\kappa \leq 1$.  We sum
  the terms for the different values of $j$ to obtain the corresponding density (the sum converges
  because $V_{g-2-j,2j+3}/V_g = \O{1}$ by \cite[Lemma 3.2]{mirzakhani2013}). We can similarly write the
  terms arising from the other realizations of rank $1$, which are obtained by adding some Dirac
  distributions. This yields additional terms, multiple of
  \begin{equation*}
    \int_{\star}
    f_1(\ell_i) \, \ell_j \delta(\ell_j-\ell_k)
    \frac{\d \ell_1 \d \ell_2 \d \ell_3}{\d \ell}
    \quad \text{for } \{i,j,k\}=\{1,2,3\}
  \end{equation*}
  but also, when we glue one or two cylinders to the pair of pants $\mathbf{P}$,
  \begin{equation*}
    \int_{\star}
    \prod_{i=1}^3 f_1(\ell_i) \prod_{i \in I} \1{[0,\kappa]}(\ell_i) 
    \frac{\d \ell_1 \d \ell_2 \d \ell_3}{\d \ell}
    \quad \text{for } I \subseteq \{1,2,3\},
  \end{equation*}
  or even a mixture of both, for instance if we glue two components of $\mathbf{P}$ to the same
  cylinder.  We notice that $f_1$ and $f_1 \1{[0,\kappa]}$ both satisfy
  \begin{equation*}
    \frac{f(\ell)}{\sinh \div{\ell}} = p(\ell)
    + \O[\kappa]{e^{- \ell/2}}
  \end{equation*}
  for the (constant) polynomial $p=2$ or $0$ respectively, because
  $\1{[0,\kappa]} = \O[\kappa]{e^{-\ell/2}}$. Then, \cite[Theorem 8.1]{Ours1}
  applied to the local type $\type$ allow us to conclude that these densities are indeed
  Friedman--Ramanujan in the weak sense. Since we systematically integrate only functions that are
  products of $f_1$ and $f_1 \1{[0,\kappa]}$, i.e. there are $\leq 2^3$ possible products, and by
  the uniformity in \cite[Theorem 8.1]{Ours1} with respect to the local type, we further obtain that the
  resulting function lies in $\cF_w^{c,c}$ for a constant $c$, and that its norm is $\O[\kappa]{1}$.

  The computation is the same for the average
  $\avQ[\rho_j\type^\partial]{\mu_\kappa^j \otimes (F,\mu_\tfL^3)}$, with the difference that now,
  we sometimes multiply by the indicator function
  $\1{[0, \kappa \log(g)]}(\ell) = \O{g^{\kappa/2} e^{-\ell/2}}$ coming from the tangles, and hence
  there is $\leq 3^3$ possible products. We then obtain, once again, a uniform bound on the
  Friedman--Ramanujan degree, but the norm is now $\O[\kappa]{g^{3 \kappa/2}}$. We will explain
  further why the contribution of the last terms
  $\sum_{\type' \in \mathrm{Loc}_{1,1}^{\type,\mathrm{s},\partial}} \avQ[\rho_j\type']{\mu_\kappa^j
    \otimes (F,\mu_\kappa^1,\mu_\tfL^1)}$ can be put in the error term, which is then enough to conclude in
  the case of a type $\type$ filling a pair of pants.

Let us now consider the local topology ``simple''. Similarly, we examine each term
in~\cref{p:incl_s} one after the other; first, let us treat all the terms except the terms
containing the summations over $\mathrm{Loc}_{1,1}^{2\mathrm{s}}$ and
$\mathrm{Loc}_{1,1}^{2\mathrm{s},\partial}$. We use the asymptotic expansion
\cref{theo:volume_asympt_exp} at the second order ($\ord=1$). This allows us to express the product
of the density and the function $\ell \mapsto \ell$ as a linear combination of functions
$\ell \mapsto f^{(\kappa)}_{m_1}(\ell) f^{(\kappa)}_{m_2}(\ell)$ with bounded coefficients,
where the family of functions $(f_i^{(\kappa)})_{1 \leq i \leq 9(a_2+1)}$ consists of all products
  of the functions
$$\ell \mapsto \ell^k \cosh \div{\ell} \qquad
\ell \mapsto \ell^k \sinh \div{\ell} \qquad \ell \mapsto \ell^k$$ where $0 \leq k \leq a_2$, with
the functions $1$, $\1{[0,\kappa]}$ and $\1{[0, \kappa \log(g)]}$. We check all of these
functions satisfy
  \begin{equation}
    \label{eq:cond_fi_kappa}
    \frac{f(\ell)}{\sinh \div{\ell}} = p(\ell)
    + \O[\kappa]{g^{c \kappa} (\ell+1)^c e^{- \ell/2}}
  \end{equation}
  for a polynomial $p$ and a constant $c>0$, which is enough to conclude for these terms.

  Now, the last cases we need to consider are the cases where we sum over local topological types in
  $\mathrm{Loc}_{1,1}^{2\mathrm{s}}$, $\mathrm{Loc}_{1,1}^{2\mathrm{s},\partial}$ as well as
  $\mathrm{Loc}_{1,1}^{\type,\mathrm{s},\partial}$ for a local type $\type$ filling a pair of
  pants. These terms are different because they correspond to situations where the geodesic $\gamma$
  has an intersection with one of the $\kappa$-short loops.

  Let us first look at the term $\avQ[\rho_j \type]{\mu_\kappa^j \otimes (F,\mu_\kappa)}$ in
  \cref{p:incl_s}, where $\type \in \mathrm{Loc}_{1,1}^{2\mathrm{s}}$.  As a first step, we compute
  the density appearing for the local topology $\eqc{\Sf_{1,1},(\gamma, \beta)}$, where
  $(\gamma, \beta)$ is a pair of simple loops on $\Sf_{1,1}$ that intersect exactly once.  We pick
  Fenchel--Nielsen coordinates $(\ell, \theta)$ on the once-holed torus associated to the geodesic
  $\gamma$, with the origin of twist taken for $\beta$. As a consequence, the density we wish to
  compute can be expressed as:
  \begin{equation*}
    \sum_{j=0}^{+\infty} \frac{(-1)^j}{j!} 
    \int_{0}^{+\infty} \int_{- \infty}^{+ \infty}
    \frac{x x_1 \ldots x_jV_{g-1-j,1+2j}(x,x_1, x_1, \ldots, x_j, x_j)}{V_g}
    \1{[0,\kappa]}(\ell_{\ell,\theta,x}(\beta))
    \d x \d \theta
  \end{equation*}
  where $\ell_{\ell,\theta,x}(\beta)$ denotes the length of $\beta$ in the once-holed torus of
  boundary length $x$ with Fenchel--Nielsen coordinates $(\ell, \theta)$.  By
  \cref{theo:volume_asympt_exp} with $\ord=0$, we need to compute
  \begin{equation*}
    \sum_{j=0}^{+\infty} \frac{(-1)^j I(\kappa)^j}{j!} \frac{V_{g-1-j,1+2j}}{V_g}
    \int_{0}^{+\infty} \int_{- \infty}^{+ \infty}
    \sinh \div{x}  \1{[0,\kappa]}(\ell_{\ell,\theta,x}(\beta))
    \d x \d \theta.
  \end{equation*}
  It is therefore enough to prove that the integral
  \begin{equation*}
    J_\kappa(\ell) :=     \int_{0}^{+\infty} \int_{- \infty}^{+ \infty}
    \sinh \div{x}  \1{[0,\kappa]}(\ell_{\ell,\theta,x}(\beta))
    \d x \d \theta
  \end{equation*}
  is a Friedman--Ramanujan function.

  Let us equip $\Sf_{1,1}$ with the metric associated with the parameters $(x, \ell, \theta)$. We
  denote as $p$ the length of the orthogeodesic linking the two boundaries of length $\ell$ on the
  pair of pants $\Sf_{1,1} \setminus \gamma$, and $r$ the length of $\beta$ for this metric. Then,
  \begin{equation*}
    \begin{cases*}
      \cosh \div{\theta} = \cosh \div{r} /\cosh \div{p} \\
      \cosh \div{x} = \cosh(p) \sinh^2 \div{\ell} - \cosh^2 \div{\ell}.      
    \end{cases*}
  \end{equation*}
  We perform the successive changes of variables
  $(x,\theta) \rightarrow (r,p) \rightarrow (R,P)$ for $R:=\cosh \div{r}$ and
  $P := \cosh \div{P}$, which yield
  \begin{equation*}
    J_\kappa(\ell)
    = 16 \sinh^2 \div{\ell}
    \iint_{\frac{1}{\tanh \div{\ell}} <P < R < \cosh \div{\kappa}} \frac{\d R \d P}{\sqrt{(R/P)^2-1}}.
  \end{equation*}
  We then check that the integral above converges to a finite limit as $\ell \rightarrow +\infty$,
  and that the difference between the integral and its limit is $\O[\kappa]{e^{-\ell/2}}$. This
  implies that $J_\kappa$ is a Friedman--Ramanujan function.

  Same goes when we examine the term $\eqc{\Sf_{1,1}, (\gamma, \beta, \partial \Sf_{1,1})}$ in the
  sum over $\mathrm{Loc}_{1,1}^{2\mathrm{s},\partial}$. Indeed, the expression is the same, except
  there is an additional indicator function $\1{[0, \kappa \log(g)]}(x)$. We note that
  $\theta \leq r \leq \kappa$, and hence the integral we need to compute is smaller than
  \begin{equation*}
    \int_{0}^{\kappa \log(g)}  \sinh \div{x} \d x
    \int_{-\kappa}^{\kappa} \d \theta = \O[\kappa]{g^{\kappa/2}},
  \end{equation*}
  which implies that the corresponding density is an element of $\FRrem$.

  In order to conclude, we now prove that the contributions of all other local topological types in
  $\mathrm{Loc}_{1,1}^{2\mathrm{s}}$, $\mathrm{Loc}_{1,1}^{2\mathrm{s},\partial}$ and
  $\mathrm{Loc}_{1,1}^{\type,\mathrm{s},\partial}$ are
  $\O[\kappa]{g^{c\kappa} \norminf{F(\ell) e^{\ell/2}}}$. We only provide the details for the first
  case. From the cases above, we know that the summation over $j$ is not an issue, so we shall only
  examine the case $j=0$.

  Let us fix a simple loop $\beta$ on $\Sf_{1,1}$. Then, $\beta$ is a pair of pants decomposition of
  $\Sf_{1,1}$, and therefore we can use Dehn--Thurston coordinates to enumerate all simple loops
  on~$\Sf_{1,1}$ (see \cite{penner1922}). As a consequence, the data of a simple loop $\gamma$ is
  entirely determined by the pair of integers $(i, t)$, where $i$ is the intersection number of
  $\gamma$ and $\beta$, and $t$ the number of twists of $\gamma$ around $\beta$. We denote as $x$
  the length of $\partial \Sf_{1,1}$ and $y$ the length of $\beta$. Then, the quantity we wish to
  estimate can be expressed as
  \begin{equation}
    \label{eq:proof_prop_density_inclexcl_pair_inter_1}
    \int_{0}^{+ \infty} \int_0^\kappa \int_{0}^{y}
    \sum_{(i,t)} F(\ell_{x,y,\theta}(\gamma_{i,t})) \frac{V_{g-1,1}(x)}{V_g} x \d x \d y \d \theta 
  \end{equation}
  where $(i,t)$ runs over a subset of $\Z^2$, $\gamma_{i,t}$ is the simple loop of Dehn--Thurston
  coordinates $(i,t)$, and $\ell_{x,y,\theta}(\gamma_{i,t})$ denotes its length for the metric
  $(x,y,\theta)$. Note that the twist parameter~$\theta$ is only taken in the interval $[0,y)$,
  because for any $(i,t)$, the image of $\gamma_{i,t}$ by a Dehn twist around $\beta$ corresponds to
  another term $(i,t')$ in the sum.

  We use the usual bound $V_{g-1,1}(x)/V_{g} = \O{e^{x/2}}$, and make the following observations. On
  the one hand, for small enough $\kappa$, the neighbourhood of width $1$ around $\beta$ is a
  cylinder, by the collar lemma. It follows that the number of intersections between any geodesic
  $\gamma_{i,t}$ of length $\leq L$ and $\beta$ satisfies $i \leq L/2$. Similarly, we have that
  $t \leq L/y$.  On the other hand, $i \neq 0$ because $\beta$ and $\gamma$ are two simple distinct
  loops on $\Sf_{1,1}$, and $i \neq \pm 1$ because that is the case we treated above. Hence,
  $\abso{i} \geq 2$. It is easy to prove, using arguments similar to the double-filling arguments in
  \cite[Section 8.1]{Ours1}, that this implies that $x \leq \ell_{x,y,\theta}(\gamma_{i,t})+2
  y$. Altogether, these bounds allow us to say that the quantity
  \eqref{eq:proof_prop_density_inclexcl_pair_inter_1} is bounded, in absolute value, by
  \begin{align*}
     \norminf{F(\ell) e^{\frac \ell 2}} \int_0^{L+2 \kappa} \int_0^\kappa \int_0^y \frac{L^2}{y}
      \, x \, e^y \d x \d y \d \theta 
    = \O[\kappa]{L^4 g^\kappa \norminf{F(\ell) e^{\frac \ell 2}}}
  \end{align*}
  which is enough to conclude.
\end{proof}

\section{Proof of the main result}
\label{sec:proof_29}

We now have all the elements needed to proceed to the proof of \cref{thm:29}.

\begin{proof}
  Let $\epsilon > 0$. Let us take a free parameter $\kappa \in (0,1)$ (we shall make further
  assumptions on $\kappa$ throughout the proof). Then, for $\tfL=\kappa \log g$,
  \begin{align*}
    & \Pwp{\lambda_1 \leq \frac 29 - \epsilon} \\
    & \leq
     \Pwp{\kappa \leq \lambda_1 \leq \frac 29 - \epsilon \text{ and } X \in
      \Atf}
    + \Pwp{X \notin \Atf} + \Pwp{\lambda_1 \leq \kappa}.
  \end{align*}
  We saw in \cref{lem:prob_TF} that
  $\Pwp{X \notin \Atf} = \O{\kappa^2 + g^{\frac 32 \kappa -1}}$. Furthermore,
  Mirzakhani proved in \cite[Theorem 4.7]{mirzakhani2013} that, provided that $\kappa < 0.002$,
  $\Pwp{\lambda_1 \leq \kappa}$ goes to zero as $g \rightarrow + \infty$. As a consequence, if we
  prove that, for any small $\kappa >0$,
  \begin{equation}
    \label{eq:29_add_TF}
    \lim_{g \rightarrow + \infty}
    \Pwp{\kappa \leq \lambda_1 \leq \frac 29 - \epsilon \text{ and } X \in \Atf}
    = 0
  \end{equation}
  then we obtain that 
  \begin{equation*}
    \limsup_{g \rightarrow + \infty} \Pwp{\lambda_1 \leq \frac 29 - \epsilon}
    = \O{\kappa^{2}}.
  \end{equation*}
  We will therefore be able to deduce \cref{thm:29} by letting $\kappa$ go to zero.
  
  Now, let $h_L(\ell) = h(\ell/L)$ denote the test function from \cite[Notation 3.8]{Ours1}
  constructed by dilatation from the smooth even function $h : \R \rightarrow \R_{\geq 0}$ supported
  on $[-1,1]$ such that $\hat{h}$ is non-negative on $\R \cup i [- \frac 12, \frac 12]$.
  We consider the length-scale $L:=6 \log (g)$. Let us fix an integer $m \geq 0$, to be determined
  later. Motivated by the discussion in \cite[Section 3.4]{Ours1}, we shall apply the Selberg trace
  formula to the function $H_{L,m} := \D^m h_L$, where $\D = 1/4 - \partial^2$.

  We use \cite[Lemma 3.11]{Ours1} to bound $\hat{h}_L(r_1)$ in the event that
  $\lambda_1 \leq 2/9 - \epsilon$. Indeed, we have $2/9 = 1/4 - (1/6)^2$, and hence taking
  $\alpha = 1/6$, we obtain that provided $\epsilon$ is small enough, there exists a constant
  $C_\epsilon > 0$ such that, for any $X$ satisfying $\lambda_1(X) \leq 2/9 - \epsilon$,
  \begin{equation*}
    \hat{h}_{L}(r_1(X)) \geq C_{\epsilon} \, e^{(\alpha+\epsilon) L}
    = C_{\epsilon} \, g^{1+6\epsilon}.
  \end{equation*}
  It follows that, if $\kappa \leq \lambda_1 \leq 2/9 - \epsilon$, then the Fourier transform
  $r \mapsto (1/4+r^2)^m\hat{h}_L(r)$ of $H_{L,m}$, taken at $r=r_1$, is larger than
  $C_{\epsilon}\kappa^m \, g^{1+6\epsilon}$.  As a consequence, we can write
  \begin{align*}
    \Pwp{\kappa \leq \lambda_1(X) \leq \frac 29 - \epsilon \text{ and } X \in
      \Atf} 
    & \leq \Pwp{\hat{H}_{L,m}(r_1) \, \1{\Atf} \geq
      C_\epsilon \, \kappa^m
      g^{1+6\epsilon}}  \\
     & \leq \frac{\Ewpo \big[ \hat{H}_{L,m}(r_1) \, \1{\Atf} \big]}
      {C_\epsilon \, \kappa^m \, g^{1+6\epsilon}} 
  \end{align*}
  by Markov's inequality (here, we have used the fact that $\hat{H}_{L,m}(r_1) \geq 0$, because
  $\lambda_1 \geq 0$ and, by hypothesis, $\hat{h}_{L}(r_1) \geq 0$). This reduces the problem to
  finding an integer $m$ such that, for any small enough $\kappa$,
  \begin{equation}
    \label{eq:reformulation_proof_29}
    \Ewpo \big[ \hat{H}_{L,m}(r_1) \, \1{\Atf} \big] =
    \O[\epsilon,\kappa]{g^{1+5\epsilon}}
  \end{equation}
  which shall now be our objective.

  By positivity of $\hat{H}_{L,m}$ on $\R \cup i[-1/2, 1/2]$, the expectation
  $\Ewpo \big[ \hat{H}_{L,m}(r_1) \, \1{\Atf} \big]$ is smaller than the
  expectation of the Selberg trace formula applied to $H_{L,m}$, multiplied by the indicator of
  $\Atf$. Following the proof of \cite[Lemma 3.12]{Ours1} with this additional
  indicator function, we prove that
  \begin{equation}
    \label{eq:proof_29_after_selberg}
    \begin{split}
      & \Ewpo \big[ \hat{H}_{L,m}(r_1) \, \1{\Atf} \big]\\
      & \leq  \Ewp{\sum_{\gamma \in \mathcal{G}(X)}
      \frac{\ell_X(\gamma) \, H_{L,m}(\ell_X(\gamma))}{\exp \div{\ell_X(\gamma)}}
      \1{\Atf(X)}} + \O{(\log g)^2 g}.
    \end{split}
  \end{equation}

  {
    Using the uniform counting bound on geodesics, \cite[Lemma 2.2]{Ours1},
    as in the proof of \cite[Theorem 6.1]{Ours1}, we find a constant
  $\chic$ such that
  \begin{equation*}
     \Ewp{\sum_{\substack{\gamma \in \mathcal{G}(X) \\ \chi(S(\gamma)) > \chic}}
       \frac{\ell_X(\gamma) \, H_{L,m}(\ell_X(\gamma))}{\exp \div{\ell_X(\gamma)}}\1{\Atf(X)}}
       = \O{g}.
  \end{equation*}
  Hence, the sum in \cref{eq:proof_29_after_selberg} can be reduced to geodesics filling a surface
  of absolute Euler characteristic $\leq \chic$.}

  For every filling type $\Sf$ such that $2 \leq \chi(\Sf) \leq \chic$, we use Wu--Xue's counting
  result, \cite[Theorem 2.4]{Ours1}, together with \cite[Lemma 5.21]{Ours1} for $\ord=\chi(\Sf)$ to
  obtain that
  \begin{equation}
    \label{eq:proof_29_1}
    \Ewp{\sum_{\gamma : S(\gamma) = \Sf}
      \frac{\ell_X(\gamma) \, H_{L,m}(\ell_X(\gamma))}{\exp \div{\ell_X(\gamma)}}
      \1{\Atf(X)}}
    = \O[\epsilon]{\frac{e^{(1+\epsilon)L/2}}{g^2}}
    = \O[\epsilon]{g^{1+3\epsilon}}.
  \end{equation}
  The number of such filling types is $\O{1}$. As a consequence, equations
  \eqref{eq:proof_29_after_selberg} and \eqref{eq:proof_29_1} together imply that
  $\Ewpo \big[ \hat{H}_{L,m}(r_1) \, \1{\Atf} \big]$ is smaller than
  \begin{equation}
    \label{eq:proof_29_remove_high_euler}
    \Ewp{\sum_{\gamma : \chi(S(\gamma)) \leq 1}
      \frac{\ell_X(\gamma) \, H_{L,m}(\ell_X(\gamma))}{\exp \div{\ell_X(\gamma)}}
      \1{\Atf(X)}} + \O[\epsilon]{g^{1+3\epsilon}}.
  \end{equation}

  We shall now split this sum by local topological type. In order to do so in a
  controlled way, we apply \cref{lem:TF_curves}. This gives us a set
  $\mathrm{Loc}_g^{\kappa}$ of local topologies such that, as soon as
  $X \in \Atf$, all closed geodesics shorter than $L=6 \log g$ filling a
  pair of pants or once-holed torus are locally equivalent to a type in
  $\mathrm{Loc}_g^{\kappa}$, and
  \begin{equation}
    \label{eq:proof_29_5}
    \# \mathrm{Loc}_g^{\kappa} = \O[\kappa]{(\log g)^{c_\kappa}}
    = \O[\kappa,\epsilon]{g^\epsilon}.
  \end{equation}
  The bound \eqref{eq:proof_29_remove_high_euler} can then be rewritten as
  \begin{equation}
    \label{eq:proof_29_2}
    \Ewpo \big[ \hat{H}_{L,m}(r_1) \, \1{\Atf} \big]
    \leq \sum_{\type \in \mathrm{Loc}_g^{\kappa} \cup \{\mathbf{s}\}} \mathrm{Av}_{g,m,\kappa}(\type) + \O[\kappa,\epsilon]{g^{1+3\epsilon}}
  \end{equation}
  where for any $\type \in \mathrm{Loc}_g^{\kappa} \cup \{\mathbf{s}\}$,
  \begin{equation}
    \label{eq:proof_29_3}
    \mathrm{Av}_{g,m,\kappa}(\type)
    := \Ewp{\sum_{\substack{\gamma \sim \type \\ \ell_X(\gamma) \leq L}}
      \frac{\ell_X(\gamma) \, H_{L,m}(\ell_X(\gamma))}{\exp \div{\ell_X(\gamma)}}
      \1{\Atf(X)}}.
  \end{equation}

  Let $\type \in \mathrm{Loc}_g^{\kappa} \cup \{\mathbf{s}\}$.
  \begin{itemize}
  \item If $\type$ fills a once-holed torus, then by \cite[Proposition 8.8]{Ours1}, $\type$
    is a double-filling loop. We then bound naively
    \begin{equation}
      \label{eq:proof_29_1_1}
      \abso{\mathrm{Av}_{g,m,\kappa}(\type)}
      \leq \avb[\type]{\abso{H_{L,m}(\ell)} \,\ell \, e^{- \frac \ell 2}}
    \end{equation}
    and use our asymptotic expansion, \cite[Theorem 5.15]{Ours1}, at the order
    $2$ with $\epsilon'=\epsilon/2$, to write
    $\avb[\type]{\abso{H_{L,m}(\ell)} \,\ell \, e^{- \frac \ell 2}}$ as
    \begin{equation*}
       \int_{0}^{+ \infty} \abso{H_{L,m}(\ell)} e^{- \frac \ell 2}
        \paren*{\ell f_0^{\type}(\ell) + \frac{\ell
            f_1^\type(\ell)}{g}} \d \ell
      + \O[\epsilon]{\frac{\norminf{\ell H_{L,m}(\ell) e^{\frac{(1+\epsilon)\ell}{2}}}}{g^{2}}}.
    \end{equation*}
    The remainder is $\O[\epsilon,m]{g^{1+4\epsilon}}$ by definition of $H_{L,m}$ and since
    $L=6 \log(g)$.  By \cite[Proposition 8.5]{Ours1}, there exists a constant $c_1>0$ independent of
    $\type$ such that $f_0^\type$ and $f_1^\type$ belong in $\FRrem_w^{c_1}$ and their weak
    Friedman--Ramanujan norm is $\leq c_1$. It then follows directly by the definition of
    $\FRrem_w^{c_1}$ that
    \begin{equation}
      \label{eq:proof_29_s_1_1_conclusion}
      \mathrm{Av}_{g,m,\kappa}(\type) = \O[\epsilon,m]{\norminf{(\ell+1)^{c_1}
          H_{L,m}} + g^{1+4\epsilon}}
      = \O[\epsilon,m]{g^{1+4\epsilon}}.
    \end{equation}
  \item If $\type$ is simple or fills a pair of pants, we apply
    \cref{prp:FR_incl_exclu}. We obtain that there exists a constant $c_2>0$
    such that 
    \begin{equation}
      \label{eq:proof_29_0_3}
      \mathrm{Av}_{g,m,\kappa}(\type) =
      \int_{0}^{+\infty}  \D^m h_L(\ell) \, e^{- \frac \ell 2}
      \, \ell A_{\type,g}^{\kappa}(\ell) \d \ell
      + \O[\kappa,m]{g^{1+c_2 \kappa}}, 
    \end{equation}
    where $\ell \mapsto \ell A_{\type,g}^{\kappa}(\ell)$ belongs in $\cF_w^{c_2,c_2}$ and has weak
    Friedman--Ramanujan norm $\O[\kappa]{g^{c_2 \kappa}}$. We now specify the value of the parameter
    $m$ to be $m := \lceil c_2 \rceil$, so that we can use the cancellation properties of
    Friedman--Ramanujan functions. More precisely, by \cite[Proposition 3.17]{Ours1},
    \begin{align*}
       \int_{0}^{+\infty} \D^m h_L(\ell) \, e^{- \frac \ell 2}
      \ell A_{\type,g}^{\kappa}(\ell) \d \ell
      = \O[\kappa]{g^{c_2 \kappa} (L+1)^{c_2+1}}
      = \O[\epsilon, \kappa]{g^{2 \epsilon}}
    \end{align*}
    as soon as $\kappa < \epsilon/c_2$.
    Together with \eqref{eq:proof_29_0_3}, this implies that for small enough $\kappa$,
    \begin{equation}
      \label{eq:proof_29_s_0_3_conclusion}
      \mathrm{Av}_{g,m,\kappa}(\type)
      = \O[\epsilon,\kappa]{g^{1+\epsilon}}.
    \end{equation} 
  \end{itemize}
  Then, equations \eqref{eq:proof_29_5}, \eqref{eq:proof_29_s_1_1_conclusion}
  and \eqref{eq:proof_29_s_0_3_conclusion} together imply that
  \begin{equation*}
    \sum_{\type \in \mathrm{Loc}_g^{\kappa} \cup \{\mathbf{s}\}} \mathrm{Av}_{g,m,\kappa}(\type)
    = \O[\epsilon,\kappa]{g^{\epsilon+1+4\epsilon}} = \O[\epsilon]{g^{1+5 \epsilon}}.
  \end{equation*}
  Then, by \cref{eq:proof_29_2},  for small enough $\kappa$,
  \begin{equation*}
    \Ewpo \big[ \hat{H}_{L,m}(r_1) \, \1{\Atf} \big]
    = \O[\epsilon,\kappa]{g^{1+5\epsilon} + g^{1+3 \epsilon}}
    = \O[\epsilon,\kappa]{g^{1+5 \epsilon}}
  \end{equation*}
  which is exactly what was needed to conclude in light of
  \cref{eq:reformulation_proof_29}.
\end{proof}

\appendix

\section{Discarding sets of extremely small probability}
\label{s:A3}

In this appendix, we exhibit a set of extremely small probability in the moduli space, i.e. surfaces
containing a multi-curve of small lengths disconnecting the surface in many connected components.

By ``extremely small probability'', we mean probability $\cO(g^{-N})$ where we can tune the
parameters to make the exponent $N$ as large as we need.  This stands in contrast to the set of
surfaces containing tangles, which has probability of order $\kappa^2 + g^{\frac 3 2 \kappa-1}$.

The reason why sets of ``extremely small probability'' can easily be discarded in the trace method
is that they contribute a negligible amount to our trace averages. This comes from the classical
deterministic counting estimate for periodic geodesics, see e.g. \cite[Lemma 2.2]{Ours1}.  An
example of this phenomenon is the use of \cite[Proposition 6.4]{Ours1} to prove \cite[Theorem
6.1]{Ours1}. We here exhibit another event of extremely small probability, having short multi-curves
which separate the surface of genus $g$ into many components.

\begin{nota}
  For an integer $Q$ and a hyperbolic surface $X$, we denote as ${\MC}_X(Q)$ the set of multi-curves
  which separate the surface $X$ into at most $Q$ connected components.
\end{nota}

We recall that by multi-curve we mean family of simple disjoint loops which are not homotopic to one
another.  We shall estimate the probability of the following events.
 
\begin{nota}
  For an integer $Q \geq 1$, $g \geq 2$ and $0 < \kappa <1$, let  
  \begin{align*}\cB_g^{ \kappa, Q}
    :=
    \left\{ X\in \cM_g \, : \,
    \forall \gamma
    \text{ multi-curve on } X, \ell^{\mathrm{max}}(\gamma) \leq \kappa \Rightarrow
    \, \gamma \in \MC_X(Q)
    \right\}
  \end{align*}
  where $\ell^{\mathrm{max}}$ denotes the length of the longest component of the multi-curve $\gamma$.
\end{nota}

We shall prove the following bound on the probability of the event $\cB_g^{\kappa,Q}$.

\begin{prp}\label{l:PBN}
  For any $Q \geq 1$, $g \geq 2$ and $0 < \kappa < 1$, 
  \begin{align}\label{e:PBN}
    1-\Pwp{ \cB_g^{ \kappa, Q}} = \O[\kappa, Q]{\frac{1}{g^{Q-1}}}.
 \end{align}
\end{prp}

In other words, the events $\cB_{g}^{\kappa,Q}$ are events of extremely small probability: we can
adjust the parameter $Q$ to match any desired rate of decay.

In order to prove this bound, let us introduce a counting function $\cY_{\kappa,Q}$, defined for a
compact hyperbolic surface $X$ as
\begin{align*}
  \cY_{\kappa,Q}(X)
  =\sum_{j=1}^{+\infty}
  \Ninj[,j,Q](X)
\end{align*}
where $\Ninj[,j,Q](X)$ counts unordered families of $\kappa$-short loops with $j$ elements, which
disconnect the surface in at most $Q$ components:
\begin{equation}
  \label{eq:NinjQ}
  \Ninj[,j,Q](X) := \frac{1}{j!}
  \sum_{\substack{(\gamma_1, \ldots, \gamma_j) \in {\MC}_X(Q) }}
  \prod_{i=1}^j \bbbone_{[0, \kappa]}(\ell_X(\gamma_i)).
\end{equation}
More generally, for any $\beta>0$, we define a weighted version
\begin{align*}
  \cY_{\kappa,Q, \beta}(X)
  = \sum_{j=1}^{+\infty} \beta^j \Ninj[,j,Q](X).
\end{align*}
The sum actually stops at $j=3g-3$, but we want to estimate the expectation of $\cY_{\kappa,Q, \beta}$
uniformly as $g$ varies. We prove the following.

\begin{prp} \label{p:YN}
  For any fixed $Q \geq 1$, $0 < \kappa < 1$,  $\beta >0$ and $n \geq 0$,
\begin{align*}
  \sup_{g}\Ewpon \brac*{\cY_{\kappa,Q, \beta}} <+\infty
  \quad \text{and} \quad
  \sup_{g}\Ewpon \brac*{\cY^2_{\kappa,Q, \beta}} <+\infty.
\end{align*}
\end{prp}

\begin{proof}
  We use the upper bound \cite[equation (2.9)]{Ours1} and standard integration methods on the moduli
  space, which gives
 \begin{align}
   \label{e:L1}
   \Ewpon \brac*{\cY_{\kappa,Q, \beta}} \leq \frac{1}{V_{g, n}}
   \sum_{j=1}^{+\infty} \frac{\beta^j  I_\kappa^{j}}{j!}
   \sum_{\mathfrak{q}=1}^Q
   \sum_{\vec{g},\vec{n}}
   \# \mathrm{Orb}_{g, n}^{j,Q}(\vec{g},\vec{n}) \prod_{i=1}^{\mathfrak{q}} V_{g_i, n_i}
 \end{align}
where:
\begin{itemize}
\item the sum over $1 \leq \mathfrak{q} \leq Q$ is a sum over the number of connected components of
  the complement of $\gamma$;
\item the sum over $\vec{g}, \vec{n}$ is a sum over the vectors
  $(g_i, n_i)_{1 \leq i \leq \mathfrak{q}}$ satisfying $2 g_i-2+n_i >0$, 
  \begin{equation*}
    \sum_{i=1}^{\mathfrak{q}} n_i =2j
    \quad \text{and} \quad
    \sum_{k=1}^{\mathfrak{q}} (2g_i-2+n_i)=2g-2+n;
  \end{equation*}
\item the quantity $I_\kappa$ is defined as $I_\kappa= \int_{0}^\kappa \ell  e^{\frac{\ell}2}d\ell$;
\item $\mathrm{Orb}_{g, n}^{j,Q}(\vec{g},\vec{n})$ is the set of MCG-equivalence classes, in a
  surface $S_{g, n}$ of signature $(g, n)$, of multi-curves $\gamma=(\gamma_1, \ldots, \gamma_j)$ with $j$
  components, cutting the surface $S_{g, n}$ into $\mathfrak{q}$ numbered pieces of respective signatures
  $(g_1, n_1), \ldots, (g_{\mathfrak{q}}, n_{\mathfrak{q}})$.
\end{itemize}

We then prove that for each $\vec{g}, \vec{n}$,
$$ \# \mathrm{Orb}_{g, n}^{j,Q}(\vec{g},\vec{n}) \leq \mathfrak{q}^{2j}.$$
To see this, we define a surjective map from a subset
$\mathcal{A} \subset \{1, \ldots, \mathfrak{q}\}^{2j}$ onto
$\mathrm{Orb}_{g, n}^{j,Q}(\vec{g},\vec{n})$. Fix a family of surfaces
$(S_i)_{1 \leq i \leq \mathfrak{q}}$ of respective signatures
$(g_i, n_i)_{1 \leq i \leq \mathfrak{q}}$. Given a sequence $(i_k, i_k')_{1 \leq k \leq j}$ with
$(i_k, i_k')\in \{1, \ldots, \mathfrak{q}\}^2$, glue successively, for $k = 1, \ldots, j$, a
boundary curve of $S_{i_k}$ to a boundary curve of $S_{i_k'}$, and call this curve $\gamma_k$.  The
set $\cA$ is the subset of $\{1, \ldots, \mathfrak{q}\}^{2j}$ such that this succession of gluings
is actually possible and gives a connected surface of signature $(g, n)$. In this case, we obtain a
surface of signature $(g, n)$ together with a muticurve $(\gamma_1, \ldots, \gamma_j)$, that can be mapped to
our reference surface $S_{g, n}$ by a homeomorphism. Any multi-curve that cuts $S_{g, n}$ into
$\mathfrak{q}$ pieces of respective signatures
$(g_1, n_1), \ldots, (g_\mathfrak{q}, n_\mathfrak{q})$ is MCG-equivalent to a multi-curve obtained
this way.

Now, for fixed $(n_1, \ldots, n_{\mathfrak{q}})$, we know from \cite[Lemma 24]{nie2023} that
\begin{align}\label{e:bound_N}
  \sum_{\substack{g_1, \ldots, g_{\mathfrak{q}},
  \\ \sum_{i=1}^{\mathfrak{q}} (2g_i-2+n_i)=2g-2+n}}
  \prod_{i=1}^{\mathfrak{q}}V_{g_i, n_i}
  \leq C \Big(\frac{D}{2g-2+n}\Big)^{{\mathfrak{q}}-1} V_{g, n}
\end{align}
for universal constants $C, D$.
In particular, this is bounded above by $C V_{g, n}$.
Taking into account all the possibilities for $(n_1, \ldots, n_{\mathfrak{q}})$, we lose another
factor $(2j)^{Q}$. We end up with the upper bound
\begin{align}
  \label{e:bound_EY}
\Ewpon \brac*{\cY_{\kappa,Q, \beta}} \leq C \sum_{j=1}^{+\infty} \frac{\beta^j}{j!} Q^{2j+1}(2 j)^{Q}  I_\kappa^{j}  
\end{align}
which is a convergent series.

The bound on the second moment comes from the inequality
$\cY^2_{\kappa,Q, \beta}\leq \cY_{\kappa,Q, 4\beta}$, together with the previous case.
\end{proof}
 
We now conclude to the proof of Proposition \ref{l:PBN}, by the same method.

\begin{proof}[Proof of Proposition \ref{l:PBN}]
  For a hyperbolic surface $X$, define ${\MC}^{(1)}_X(Q)\subset {\MC}_X(Q)$ to be the set of
  multi-curves which separate $X$ into {\emph{exactly}} $Q$ components, and $\cY^{(1)}_{\kappa,Q}$
  the random variable
 \begin{align*}
   \cY^{(1)}_{\kappa,Q}(X)
   &=\sum_{j=1}^{+\infty} \frac1{j!}
     \sum_{\substack{(\gamma_1, \ldots, \gamma_j)\in {\MC}^{(1)}_X(Q)}}
    \prod_{i=1}^j\bbbone_{[0, \kappa]}(\ell_X(\gamma_i))
\end{align*}
which counts the total number of such multi-curves with maximal length $\leq \kappa$.

Using the bound~\eqref{e:bound_N} with $\mathfrak{q}=Q$, we obtain
\begin{align*}  
  \Ewpo \brac*{\cY^{(1)}_{\kappa,Q}}
  \leq \frac{D^Q}{g^{Q-1}} \sum_{j=1}^{+\infty} \frac{Q^{2j}(2 j)^{Q}  I_\kappa^{j}}{j!}
  = \O[\kappa, Q]{\frac{1}{g^{Q-1}}}.
\end{align*}
This yields the announced results, because $\cY^{(1)}_{\kappa,Q}\geq 1$ on the complement of $\cB_g^{ \kappa, Q}$.
 
 \end{proof}

 We furthermore add another useful tail estimate on the sum in the definition of $\cY_{\kappa,Q}$.

 \begin{lem}
   \label{lem:tail_j}
   For any $Q \geq 1$, $0 < \kappa < 1$ and $\ord \geq 0$, any large enough $g$,
   \begin{equation*}
     \Ewp{\sum_{j >  \log g}^{+\infty} \Ninj[,j,Q]}
     = \O[\kappa,Q,\ord]{\frac{1}{g^\ord}}.
   \end{equation*}
 \end{lem}

 \begin{proof}
   We simply use the same proof as \eqref{e:bound_EY}, now joined with the observation that the tail
   $\sum_{j > \log g}^{+\infty} (2j)^Q (Q^2I_\kappa)^j/j!$ goes to zero
   faster than any power of $g$ as $g \rightarrow + \infty$.
 \end{proof}


\section{Dependency of constants in the number of components}

In the inclusion-exclusion performed in this article, we need to understand the dependency of
certain constants more finely in order to cope with the fact that the local types we consider have
$j \gg 1$ copies of the local type simple. We explicit these bounds here.

\subsection{Constants in \cite{anantharaman2022}  }
 
Let us provide a more explicit version of the main result of~\cite{anantharaman2022}, Theorem 1.1,
where we provide asymptotic expansions for Weil--Petersson volume polynomials.

\begin{thm}
  \label{theo:volume_asympt_exp}
  For any integers $g \geq 0$, $n \geq 1$ such that $2g-2+n>0$, there exists a family of
  $n$-variable even polynomial functions $(P_{g,n}^{(\ord,V_\pm)})_{\ord, V_\pm}$, with $\ord \geq 0$ and
  $V_+ \sqcup V_- \subseteq \{1, \ldots, n\}$, such that for any integer $N \geq 0$ and any length
  vector $\x \in \R_{\geq 0}^n$,
  \begin{equation}\label{e:OKn}
     \Big|  \frac{x_1\ldots x_nV_{g,n}(\x)}{V_{g,n}}
    - F_{g,n}^{(\ord)}(\x) \Big|
    = \O[\ord,n]{\frac{(\|\x\|+1)^{3\ord+1}}{(g+1)^{\ord+1}} \exp \div{x_1+\ldots
        +x_n}}
  \end{equation}
  where
  \begin{equation*}
    F_{g,n}^{(\ord)}(\x)
    := \sum_{V_+ \sqcup V_- \subseteq \{1, \ldots,n \}}
    P_{g,n}^{(\ord,V_\pm)}(\x) \prod_{i \in V_+} \cosh
    \div{x_i}
    \prod_{i \in V_-} \sinh \div{x_i}.
  \end{equation*}
  Furthermore, we have the following.
  \begin{enumerate}
  \item The leading-order term is given explicitly by
    \begin{equation}
      F_{g,n}^{(0)}(\x)
      = 2^n \prod_{i=1}^n \sinh \div{x_i}
    \end{equation}
    i.e. the only non-zero polynomial is the one corresponding to $V_- = \{1, \ldots, n\}$ and
    $V_+=\emptyset$, and is equal to the constant polynomial $2^n$.
  \item The polynomial $P_{g,n}^{(\ord,V_\pm)}$ is even in the variables $(x_i)_{i\in V_-}$ and odd
    in  $(x_i)_{i\notin V_-}$.
\item The total degree of $P_{g,n}^{(\ord,V_\pm)}$ in the variables $(x_i)_{i\in V_+\sqcup V_-}$ is
  $\leq 2\ord$. The partial degree with respect to each $x_i$ with $i\in V_0$ can be bounded by
  a quantity $a_{\ord+1}$.
\item The coefficients of $P_{g,n}^{(\ord,V_\pm)}$ can be written as linear combinations
  (independent of $g$) of the $c_{g,n}(\alpha)/V_{g,n}$ for multi-indices $\alpha$ such that
  $\sup_{1 \leq i \leq n} \alpha_i \leq 2\ord+a_{\ord+1}$.
\item There exists $\tilde{a}_\ord$ such that the constant in \eqref{e:OKn} and the
  coefficients of the polynomials $P_{g,n}^{(\ord,V_\pm)}$ are bounded by $\tilde a_\ord^n$,
  uniformly in $g$.
  \end{enumerate}
\end{thm}

In applications to our paper, the number of boundary components $n$ will be set to be
$n= n_{\Sf}+2j$, with $n_{\Sf}$ fixed and $j$ arbitrary. This is due to ``adding'' $j$ copies of the
types simple to a local type, as needed for the inclusion-exclusion (see \cref{sec:inclusion-exclusion}).

\begin{proof}
  The expression, and the first three points, are explicitly addressed in \cite{anantharaman2022}
  (we multiplied the result of \cite[Theorem 1.1]{anantharaman2022} by $x_1\ldots x_n$ to make it
  directly usable for our purposes).
  However, the last point requires to follow the proof of \cite[Theorem 1.1]{anantharaman2022} to
  check how all the constants $C_n$, present in \cite{anantharaman2022}, depend on the number $n$ of
  boundary components (both for the constant in \eqref{e:OKn} and the coefficients of the
  polynomials).

  Thanks to \cite[Lemma 24]{nie2023}, the constants $C_n$ in \cite[eq (20)]{anantharaman2022} and
  \cite[Lemma 2.4]{anantharaman2022} may be seen to be uniform in $n$.  In \cite[Lemma
  2.8]{anantharaman2022}, the constant $C_n$ may be of order at most $2^n$, because the number of
  terms in the sum \cite[eq (24)]{anantharaman2022} is $2^n$.  This leads to the conclusion that the
  constant $C_{n, \ord}$ in \cite[Theorem 4.2]{anantharaman2022} can be bounded by $2^{2n\ord}$.

  After establishing \cite[Theorem 4.2]{anantharaman2022}, the proof of Theorem
  \ref{theo:volume_asympt_exp} is done in \S 5.4 in \cite{anantharaman2022}, and consists in
  applying \cite[Lemma 5.3]{anantharaman2022} to the function
  $f(\alpha):=c_{g,n}(\alpha) / V_{g,n}$.  It is known that
  $$0\leq c_{g,n}(\alpha)\leq c_{g,n}(0) =V_{g,n},$$
  see e.g. \cite[Lemma 3.1]{mirzakhani2013}, so we mainly need to check how \cite[Lemma
  5.3]{anantharaman2022} depends on $n$.
 
  When applying \cite[Lemma 5.3]{anantharaman2022} to bound the error term, we must take
  $|m|=2\ord+1$, $p=\ord+1$, $a=a_{\ord+1}$, so there arises a constant $C_{n, a, p, 2\ord}$ bounded
  by the quantity $2^{\ord+1+n}a_{\ord+1}^{n} (4n)^{\ord+1 }n^{2\ord+1}$. Hence, the error term in
  \cite[Lemma 5.3]{anantharaman2022} is less than
 $$2^{2n\ord}2^{\ord+1+n}a_{\ord+1}^{n} (4n)^{\ord+1 }n^{2\ord+1},$$ giving the desired control for
 the constant in \eqref{e:OKn}.
 
 We now turn to the control on the coefficients of the polynomial $P_{g,n}^{(\ord,V_\pm)}$.  In
 \cite[Lemma 5.3]{anantharaman2022}, the ``linear combinations of the values
 $\delta^{\mathbf{m}}f(\alpha)$'' may be checked to have coefficients bounded by $n^{\ord+1}$.
 Given that $|\mathbf{m}|\leq 2\ord$ and $|\alpha|\leq a_{\ord+1}$, the number of possible terms is
 at most $(a_{\ord+1}+1)^{n} (\ord+1)^n$.  Finally
 $|\delta^{\mathbf{m}}f(\alpha)| \leq  2^{2\ord} \sup_{\alpha}|f(\alpha)|=
 2^{2\ord} f(0)$, where $f(0)=1$.  The outcome is that each coefficient of $P_{g,n}^{(\ord,V_\pm)}$
 is bounded by
 $$n^{\ord+1} 2^{2\ord}  (a_{\ord+1}+1)^{n} (\ord+1)^n$$
 which implies our claim.
\end{proof}

\subsection{Rank of a realization}

We now discuss a technical estimate which will be useful when computing averages in which the
counting functions $\Ninj[,j,Q]$ appear.  We recall that for a realisation $\mathfrak{R}$ of a
filling type into a surface of genus $g$, the notion of ``rank'' defined in
\cite[Section~5.4.1]{Ours1} to describe the rank of a realization is the height in the asymptotic
expansion in powers of $1/g$ at which it appears.  This allows us to truncate sums over all
realizations to those of rank $< \ord$, as done in \cite[Lemma 5.21]{Ours1}. We here adapt this result
to the set $R_{g,Q}(\rho_j\Sf)$ of realisations of~$\Sf$ and $j$ cylinders in a surface of genus $g$
such that the $j$ cylinders separate it in at most $Q$ connected components.
\begin{lem}
  \label{lem:limit_rank_Q}
  There exists a universal constant $D>0$ such that, for any filling type~$\Sf$, any integers
  $\ord \geq \chi(\Sf)$, $Q\geq 0$, any large enough $g$, any $0 \leq j \leq \log(g)$,
  \begin{equation}\label{e:upper-rank}
    \frac{V_{g_\Sf,n_\Sf}}{V_g}
    \sum_{\substack{\mathfrak{R} \in R_{g, Q}(\rho_j \Sf) \\ \mathfrak{r}(\mathfrak{R})\geq
        \ord}}
    \prod_{\substack{1 \leq i \leq \mathfrak{q} \\\chi_i>0}} V_{g_i, n_i}
    \leq 2^{(n_\Sf+2j)(Q+n_\Sf)}  \frac{D^{Q+n_\Sf} }{g^\ord}.
  \end{equation}
\end{lem}

  \begin{rem}\label{r:rank-lc}
    The rank of a realisation is always greater than $\chi(\Sf)$, and the argument leading to
    \eqref{e:upper-rank} also shows that for any realisation
    $\mathfrak{R}\in R_{g, Q}(\rho_j \Sf) $, we have for each individual term
  \begin{equation}\label{e:rank-lc}
    \frac{V_{g_\Sf,n_\Sf}}{V_g}
    \prod_{\substack{1 \leq i \leq \mathfrak{q} \\\chi_i>0}} V_{g_i, n_i}
    \leq   \frac{D^{Q+n_\Sf}}{g^{\chi(\Sf)}}.
  \end{equation}
  
  \end{rem}
\begin{proof}
  The proof is identical to the proof of \cite[Lemma 5.21]{Ours1}, but one needs to follow the
  dependency on $j$ of the implied constants. In the first lines of proof, there is an implied
  constant $\O[\Sf]{1}$ which is the number of partitions of $\partial \Sf$.  In the current
  context, this is replaced by the number of partitions of $\partial \Sf \sqcup \{1,2\}^j$ into at most
  $Q+n_\Sf$ sets. This number is bounded by $2^{(n_\Sf+2j)(Q+n_\Sf)}$. A second constant
  $\O[\Sf]{1}$ counts the number of subsets of $\{1, \ldots, \mathfrak{q}\}$ with
  $\mathfrak{q}=Q+n_\Sf$, which is less than $2^{Q+n_\Sf}$.

  Using \cite[Lemma 24]{nie2023}, we can prove that the implied constant in
  \cite[equation (5.9)]{Ours1} may be bounded by $C D^{k-1}$ for some $D>0$, where $k$ is the number
  of connected components of $S_g\setminus \rho_j \Sf$ for this realisation. Since we are only
  counting realisations such that $k\leq Q+n_\Sf$, the implied constant in
  \cite[equation (5.9)]{Ours1} and the following lines is at most $D^{Q+n_\Sf}$.

  In \cite[equation (5.10)]{Ours1}, $\mathfrak{q}$ now varies from $2$ to $Q+n_\Sf$ and $n_{j^+}$
  varies from $1$ to $n_\Sf+2j-1$. In \cite[equation (5.11)]{Ours1}, we thus need to understand the
  dependency on $n$ of the bound $\O[\ord, n]{.}$, with $n\leq n_\Sf+2j-1$. Using \cite[Theorem
  1.8]{mirzakhani2015}, we can see that this upper bound is uniform as long as $n^2=o(g)$, more than
  we need to treat $n\leq n_\Sf+2j-1$ with $n_\Sf$ fixed and $j\leq \log g$.
\end{proof}


\bibliographystyle{plain}
\bibliography{bibliography}

\begin{thebibliography}{10}

\bibitem{anantharaman2022}
Nalini Anantharaman and Laura Monk.
\newblock A high-genus asymptotic expansion of {{Weil}}\textendash{{Petersson}}
  volume polynomials.
\newblock {\em Journal of Mathematical Physics}, 63(4):043502, 2022.

\bibitem{Ours1}
Nalini Anantharaman and Laura Monk.
\newblock Friedman--ramanujan functions in random hyperbolic geometry and
  application to spectral gaps {I}.
\newblock {\em https://arxiv.org/abs/2304.02678}, 2023.

\bibitem{Moebius}
Nalini Anantharaman and Laura Monk.
\newblock A {Moebius} inversion formula to discard tangled hyperbolic surfaces.
\newblock {\em arXiv:2401.01601}, 2023.

\bibitem{Ours2}
Nalini Anantharaman and Laura Monk.
\newblock Friedman--ramanujan functions in random hyperbolic geometry and
  application to spectral gaps {II}.
\newblock {\em https://arxiv.org/abs/2502.12268}, 2025.

\bibitem{buser1992}
Peter Buser.
\newblock {\em Geometry and {{Spectra}} of {{Compact Riemann Surfaces}}}.
\newblock {Birkh\"auser}, {Boston}, 1992.

\bibitem{friedman2003}
Joel Friedman.
\newblock A proof of {{Alon}}'s second eigenvalue conjecture.
\newblock {\em Proceedings of the thirty-fifth annual ACM symposium on Theory
  of computing}, pages 720--724, 2003.

\bibitem{huber1974}
Heinz Huber.
\newblock {\"Uber den ersten Eigenwert des Laplace-Operators auf kompakten
  Riemannschen Fl\"achen}.
\newblock {\em Commentarii mathematici Helvetici}, 49:251--259, 1974.

\bibitem{lipnowski2021}
Michael Lipnowski and Alex Wright.
\newblock Towards optimal spectral gaps in large genus.
\newblock {\em Ann. Probab.}, 52(2):545--575, 2024.

\bibitem{mirzakhani2013}
Maryam Mirzakhani.
\newblock Growth of {{Weil--Petersson}} volumes and random hyperbolic surfaces
  of large genus.
\newblock {\em Journal of Differential Geometry}, 94(2):267--300, 2013.

\bibitem{mirzakhani2015}
Maryam Mirzakhani and Peter Zograf.
\newblock Towards large genus asymptotics of intersection numbers on moduli
  spaces of curves.
\newblock {\em Geometric and Functional Analysis}, 25(4):1258--1289, 2015.

\bibitem{monk2021a}
Laura Monk and Joe Thomas.
\newblock The tangle-free hypothesis on random hyperbolic surfaces.
\newblock {\em International Mathematics Research Notices}, rnab160, 2021.

\bibitem{nie2023}
Xin Nie, Yunhui Wu, and Yuhao Xue.
\newblock Large genus asymptotics for lengths of separating closed geodesics on
  random surfaces.
\newblock {\em Journal of Topology}, 16(1):106--175, 2023.

\bibitem{penner1922}
Robert~C. Penner and John~L. Harer.
\newblock {\em Combinatorics of Train Tracks}.
\newblock Princeton University Press, 1922.

\bibitem{wu2022}
Yunhui Wu and Yuhao Xue.
\newblock Random hyperbolic surfaces of large genus have first eigenvalues
  greater than {$\frac{3}{16}-\epsilon$}.
\newblock {\em Geometric and Functional Analysis}, 32(2):340--410, 2022.

\end{thebibliography}

\end{document}